\documentclass[10pt]{amsart}

\usepackage{lineno}
\usepackage{tikz}
\usepackage{pgfplots}
\usepackage{xcolor}
\usepackage{enumitem}
\setlist[enumerate,1]{label=(\roman*)}
\usepackage[margin=2.5cm]{geometry}
\usepackage{color}
\usepackage{amsmath,amssymb,amsthm}
\usepackage{dsfont}
\usepackage{fancyhdr}
\usepackage{psfrag}
\usepackage{array}
\usepackage{graphicx}
\usepackage{caption}
\usepackage{subcaption}
\usepackage{hyperref}
\usepackage{fancybox}
\usepackage{color}
\usepackage{amsfonts}
\usepackage{stmaryrd}
\usepackage{comment}
\usepackage{subfiles}
\usepackage{pst-all}
\usepackage{stackengine}
\usepackage[normalem]{ulem}

\newtheorem{theorem}{Theorem}
\newtheorem{lemma}[theorem]{Lemma}
\newtheorem{proposition}[theorem]{Proposition}

\newtheorem{corollary}[theorem]{Corollary}

\newtheorem{remark}[theorem]{Remark}

\newcommand\mb{\mathbf}
\newcommand\mc{\mathcal}

\newcommand\R{\mb{R}}

\newcommand{\esp}{\mb{E}}

\newcommand{\ud}{\mathrm{d}}

\title{On the quadratic barycentric transport problem}
\date{\today}
\author{Nathael Gozlan}
\address{N. G. : Université Paris Cité, CNRS, MAP5, F-75006 Paris}
\email{nathael.gozlan@u-paris.fr}

\author{Thibaut Le Gouic}
\address{T. Le G. : Centrale Marseille, I2M, UMR 7373, CNRS, Aix-Marseille univ., Marseille, 13453, France}
\email{thibaut.le\_gouic@math.cnrs.fr}

\author{Paul Marie Samson}
\address{P.-M. S. : Univ Gustave Eiffel, UPEM, Univ Paris Est Creteil, CNRS, F-77447 Marne-la-Vallée, France}
\email{paul-marie.samson@univ-eiffel.fr}


\thanks{The authors are supported by a grant of the Agence nationale de la recherche (ANR), Grant ANR-23-CE40-0017 (Project SOCOT)}

\begin{document}

\begin{abstract}
We investigate the structure of optimal transport plans, dual optimizers, and geodesic paths for the quadratic barycentric transport problem.
\end{abstract}
\maketitle

\section{Introduction}
Let $\mu,\nu$ be two probability measures on $\R^d$ with finite second moments.
The optimal quadratic barycentric transport cost between $\mu$ and $\nu$ is defined by
\[
\overline{\mc{T}}_2(\nu|\mu):=\inf_{X\sim \mu, Y\sim \nu}\esp(\|\esp (Y-X|X)\|^2)
\]
where the infimum runs over all couples of random vectors $(X,Y)$ with $X \sim \mu$ and $Y\sim \nu$, and where in all the paper $\|\,\cdot\,\|$ denotes the standard Euclidean norm on $\R^d.$ This variant of the Monge-Kantorovich transport problem has been first considered in \cite{gozlan2017kantorovich} and \cite{GRST18} in the context of the concentration of measure phenomenon for convex functions. It has then been involved in several directions such as: an approach to the Caffarelli contraction theorem \cite{gozlan2020mixture, FGP20} ; a notion of Wasserstein barycenter \cite{NEURIPS2021_70d5212d}; a connection with the question of geodesic extrapolation in the Wasserstein space \cite{GNT25}. The transport cost $\overline{\mc{T}}_2$ enters the more general family of weak optimal transport problems also introduced in \cite{gozlan2017kantorovich}. See also \cite{ABC19, BBP19} for general results on this problem, and \cite{BP-survey, Beiglbock-Fundamental} for an up-to-date presentation of the different applications of the weak optimal transport problem. 

The aim of this paper is to explore further the properties of barycentric optimal transport and, in particular, to obtain an alternative formulation for $\overline{\mc{T}}_2$ in the spirit of the Benamou-Brenier formula \cite{Benamou-Brenier} for the classical Wasserstein distance $W_2$ defined by
\[
W_2^2(\mu,\nu) = \inf_{X\sim \mu,Y\sim \nu}\esp(\|Y-X\|^2).
\]
To state our main results (gathered in Theorem \ref{thm:BBbar} below), let us fix some sufficiently rich filtered probability space $(\Omega,\mathcal{F},\mathbb{P})$ (rich enough to support the existence of a Brownian motion).
As we shall prove,
\begin{equation}\label{eq:BBbar-intro}
\overline{\mc{T}}_2(\nu|\mu)=\inf\mathbb{E}\int_0^1\|v_t\|^2\ud t,
\end{equation}
where the infimum is taken over all progressively measurable processes  $(v_t)_{t\in [0,1]}$  with the following constraints: there exists an $\mathcal{F}_0$-measurable random vector $X_0\sim \mu$ and an $\mc{F}$-martingale $(M_t)_{t\in [0,1]}$ such that the process 
\[
X_t = X_0+\int_0^t v_s\,\ud s +M_t-M_0,\qquad t\in [0,1],
\]
or for short $\ud X_t = v_t\,\ud t+ \ud M_t$ satisfies  that $X_1$ has law $\nu$. Formula \eqref{eq:BBbar-intro} shows that the barycentric quadratic transport problem is a particular case of the so called semimartingale transport problem introduced by Tan and Touzi \cite{Tan-Touzi} (a framework extending \cite{Mikami, Mikami-Thieullen} ; see also \cite{BP-survey} for the link with weak optimal transport). In comparison, for the $W_2$ distance and assuming that $\mu$ is absolutely continuous with respect to Lebesgue, it holds
\begin{equation}\label{eq:BB-intro}
W_2^2(\mu,\nu)=\inf\mathbb{E}\int_0^1\|v_t\|^2\ud t,
\end{equation}
where the infimum runs this time over all $(v_t)_{t\in [0,1]}$ such that there exists $X_0 \sim \mu$ with $X_0+\int_0^1 v_s\,ds \sim \nu$.

More can be said about the optimizers in \eqref{eq:BBbar-intro}. Recall that, according to \cite{gozlan2020mixture}, the transport cost $\overline{\mc{T}}_2(\nu|\mu)$ admits the following interpretation:
\[
\overline{\mc{T}}_2(\nu|\mu) = \inf_{\eta \leq_c \nu} W_2^2(\mu,\eta),
\]
where $\leq_c$ denotes the convex order between probability measures ($\eta \leq_c \nu$ means that $\int h\,\ud \eta \leq \int h\,\ud \nu$, for all convex function $h$ on $\R^d$). Moreover the infimum is reached at a unique point $\bar{\mu} \leq_c\nu$ called the backward projection of $\mu$ onto the convex set of measures dominated by $\nu$ in the convex order (a notion also studied in depth in \cite{Alfonsi-Jourdain, Kim-Ruan}). Furthermore, there exists a continuously differentiable convex function $\varphi:\R^d \to \R$ such that $\nabla \varphi_\#\mu=\bar{\mu}$. With these notions, equality in \eqref{eq:BBbar} is achieved for any process of the form
\begin{equation}\label{eq:Xopt-intro}
X_t= (1-t)X_0+t\nabla \varphi(X_0) + M_t-\nabla\varphi(X_0),\qquad t\in [0,1]
\end{equation}
(for which  $v_t=\nabla \varphi(X_0)-X_0$, $t\in [0,1]$) with
 $(M_t)_{t\in [0,1]}$ any $\mathcal{F}$-martingale such that $M_0=\nabla\varphi(X_0)$ a.s. and $M_1\sim \nu$ (and such martingales exist). In the case of the $W_2$ distance, and assuming that $\mu$ is absolutely continuous with respect to Lebesgue, it is well known that the unique optimizer in \eqref{eq:BB-intro} is given by 
\begin{equation}\label{eq:McCann}
X_t= (1-t)X_0+tT(X_0),\qquad t\in [0,1]
\end{equation}
where $T:\R^d\to\R^d$ is the Brenier transport map \cite{Brenier1, Brenier2} sending $\mu$ onto $\nu$ and is defined, for almost every $x$, by $T(x)=\nabla F(x)$ for some convex function $F$. The path $\mu_t^{W_2} = \mathrm{Law}((1-t)X_0+tT(X_0))$, $t\in [0,1]$, is known as McCann's interpolation between $\mu$ and $\nu$ \cite{McCann} and is actually a constant speed geodesic in the sense that
\[
W_2(\mu_s^{W_2},\mu_t^{W_2}) = |t-s|W_2(\mu,\nu),\qquad \forall s,t \in [0,1].
\]
Something similar also holds in the barycentric framework: denoting by $\mu_t$ the law of $X_t$ defined at \eqref{eq:Xopt-intro}, one gets
\begin{equation}\label{eq:geodbar-intro}
\overline{\mc{T}}_2(\mu_t|\mu_s) = (t-s)^2 \overline{\mc{T}}_2(\nu|\mu),\qquad \forall 0\leq s\leq t \leq 1.
\end{equation}

A very classical property of McCann's interpolation \eqref{eq:McCann} is the following non-crossing property of trajectories: if for some $t\in (0,1)$, $(1-t)x_0+tT(x_0) = (1-t)x_0'+tT(x_0')$, then $x_0=x_0'$. A simple probabilistic consequence of this non-crossing property is that the process $X$ defined at \eqref{eq:McCann} is a (time inhomogenous) Markov process. As we shall see, the same is true in the barycentric case: if $X$ is defined by \eqref{eq:Xopt-intro} with $M$ being a Markovian martingale, then $X$ is also a Markov process. 

To prove that $X$ is Markov under the assumptions above, a crucial step consists in proving that the knowledge of $X_t$, for a given $t$, gives access to the drift $X_0-\nabla \varphi(X_0)$ (in other words, this drift is $\sigma(X_t)$-measurable) and then to $M_t$, which is the only thing that is needed to predict the future evolution of the process. We give two different proofs of this fact. The first proof relies on the construction of a paving of the space $\R^d$ into convex cells that trap all the martingales $M$ connecting $\bar{\mu}$ to $\nu$. The existence of this martingale invariant convex paving, reminiscent of a result by De March and Touzi \cite{DeMT19}, is obtained in Section  \ref{sec:paving} by projecting the flat pieces of the graph of a convex optimizer for the dual problem for $\overline{\mc T_2}(\nu|\mu)$. A second proof of the $\sigma(X_t)$-measurability of $X_0-\nabla \varphi(X_0)$ is sketched in Remark \ref{rem:alternative} and relies on the analysis of the equality cases in \eqref{eq:geodbar-intro}.

The paper is organized as follows. Section \ref{Sec:Dualpot} recalls several known results about barycentric quadratic optimal transport that will be used in the paper. We will recall in particular the dual formulations of $\overline{\mc{T}}_2$, the existence of dual optimizers, and the notions of backward and forward projections. Section \ref{sec:paving} will focus on the martingales transporting $\bar{\mu}$ to $\nu$. Using a solution of the dual problem, we will construct a convex paving of $\R^d$ that is stable for all martingales sending $\bar{\mu}$ to $\nu$. Section \ref{sec:benamou} will contain the proof of our main result Theorem \ref{thm:BBbar}. The proof of the Markov property of the optimal processes $X$ mentioned above will rely on the convex paving obtained in the preceding section. Section \ref{sec:forward-proj} will focus on the martingales sending $\mu$ to one of its forward projections $\tilde{\nu}$ (whose definition will be recalled in Section \ref{Sec:Dualpot}). We will in particular establish a simple correspondence between martingales joining $\mu$ to  $\tilde{\nu}$ and those joining $\bar{\mu}$ to $\nu$. Finally, Section \ref{sec:Gaussian} will treat the case where $\mu$ and $\nu$ are Gaussian distributions of a certain type. We will describe in this case the backward projection $\bar{\mu}$, which will turn out to be also Gaussian, and the $\overline{\mc{T}}_2$-geodesics between $\mu$ and $\nu$.

\medskip

\textbf{Note.} During the preparation of this paper, we learned about a recent preprint by Alfonsi and Jourdain \cite{Alfonsi-Jourdain-preprint} that contains a full description of the backward and forward projections in the Gaussian case. Propositions \ref{Gauss} and \ref{prop-Gaussian-Commute} independently obtained in Section \ref{sec:Gaussian} of the present paper are particular cases of their Proposition 3.1 and Theorem 4.1.

\section{Dual potentials, backward and forward projections}\label{Sec:Dualpot}
This section collects several known results on the quadratic barycentric transport problem and fixes notation that will be used in the remainder of the paper. 

In all the paper, we will denote by $\mathcal{P}_2(\R^d)$, the Wasserstein space of order $2$, i.e. the space of probability measures with a finite second moment. In all the section, $\mu,\nu$ will be two given elements of  $\mathcal{P}_2(\R^d)$.

\subsection*{Quadratic barycentric transport problem.}
The quadratic barycentric optimal transport problem between $\mu$ and $\nu$ can be restated as follows:
\begin{equation}\label{eq:barT-coupling}
\overline{\mathcal{T}}_2(\nu|\mu) = \inf_{\pi \in \Pi(\mu,\nu)} \int \left\|x- \int y\,\ud\pi_x(y)\right\|^2\,\ud \mu(x),
\end{equation}
where we recall that $\|\,\cdot\,\|$ denotes the standard Euclidean norm on $\R^d$, and where $\Pi(\mu,\nu)$ is the set of all transport plans between $\mu$ and $\nu$ (that is, the set of probability distributions on $\R^d\times \R^d$ that admit $\mu$ as first marginal and $\nu$ as second). For $\pi \in \Pi(\mu,\nu)$, we denote by $(\pi_x)_{x\in \R^d}$ the conditional disintegration of $\pi$ with respect to its first marginal, which is such that
\[
\ud \pi(x,y) = \ud\mu(x)\ud\pi_x(y).
\]

\subsection*{Dual formulations and dual optimizers.} According to \cite{gozlan2017kantorovich}, the following Kantorovich type duality formula holds:
\begin{equation}\label{eq:duality}
    \frac{1}{2}\overline{\mathcal{T}}_2(\nu|\mu) = \sup_{f}\left\{\int Q_2f\,\ud\mu - \int f\,\ud \nu\right\},
\end{equation}
where the supremum runs over the set of all convex functions $f:\R^d\to \R\cup\{+\infty\}$, and where 
\[
Q_2f(x) = \inf_{y\in \R^d} \left\{f(y)+\frac{1}{2}\|x-y\|^2\right\},\qquad x\in \R^d.
\]
Moreover, according to \cite[Theorem 6.1]{gozlan2020mixture}, there exists a lower semicontinuous convex function $\bar{f}: \R^d \to \R\cup\{+\infty\}$, integrable with respect to $\nu$, achieving equality in \eqref{eq:duality}. 

It will be useful to consider a slightly different Kantorovich dual problem involving the operator $P_2$ given by
\[
P_2g(y):=\sup_{x\in \R^d}\left\{g(x)-\frac12 \|x-y\|^2\right\},\qquad y\in \R^d,
\]
for any function $g:\R^d\to \R$.
The following result appeared in \cite{Kim-Ruan}.
\begin{lemma}\label{dualitybis}
    For any probability measures  $\mu$ and $\nu$ in $\mathcal{P}_2(\R^d)$, one has 
    \begin{equation}\label{eq:dualitybis}
    \frac{1}{2}\overline{\mathcal{T}}_2(\nu|\mu) = \sup_{g}\left\{\int g\,\ud \mu - \int P_2g\,\ud\nu\right\},
\end{equation}
where the supremum runs over the set of all convex functions $g:\R^d\to \R$.
Moreover, optimizers of \eqref{eq:duality} and \eqref{eq:dualitybis} are related as follows:
\begin{itemize}
    \item if $f$ is an optimizer of \eqref{eq:duality} ($f \neq \infty$) then $g =Q_2f$ is an optimizer of \eqref{eq:dualitybis},
    \item conversely,  if $g$ is an optimizer of \eqref{eq:dualitybis}, then $f =P_2g$ is an optimizer of \eqref{eq:dualitybis}.
\end{itemize}
\end{lemma}
Note that since \cite[Theorem 6.1]{gozlan2020mixture}) ensures the existence of an optimizer in \eqref{eq:duality} which is $\nu$-integrable, Lemma \ref{dualitybis} provides the existence of an optimizer in \eqref{eq:dualitybis}. For the sake of completeness, we include the short proof of this result.
\begin{proof}
This lemma is a simple consequence of the following observations. First, if $f$ is convex, then $Q_2f$ is also convex as an infimum convolution of two convex functions, and if $g$ is convex, then $P_2g$ is a convex function as a supremum of convex functions. Secondly, for any functions $f$ and $g$, one has $Q_2 f\leq P_2 Q_2 f\leq f$ and $P_2g\geq Q_2P_2 g\geq g$. 

Let $\bar f:\R^d\to \R\cup\{+\infty\}$, $\bar f\neq +\infty$, be an optimizer in \eqref{eq:duality} and set $\bar g =Q_2\bar f$. By choosing $y_0$ such that $\bar f(y_0)<+\infty$, one has $\bar g(x)\leq \bar f(y_0)+ \frac12 \|x-y_0\|^2<+\infty$ for any $x\in\R^d$.  Moreover,
the following inequalities lead to \eqref{eq:dualitybis} and the optimality of the convex function $\bar g:\R^d\to \R$ in \eqref{eq:dualitybis}:
\begin{align*}
     &\frac{1}{2}\overline{\mathcal{T}}_2(\nu|\mu)=\int Q_2\bar f\,\ud\mu - \int \bar f\,\ud\nu
     \leq \int \bar g\,\ud\mu - \int P_2 \bar g\,\ud\nu
     \leq \sup_{g \ \mathrm{cvx}}\left\{\int g\,\ud\mu - \int P_2g\,\ud\nu\right\}
\\&\leq \sup_{g \ \mathrm{cvx}}\left\{\int P_2 g\,\ud\mu - \int Q_2(P_2g)\,\ud\nu\right\}
\leq \sup_{f \ \mathrm{cvx}}\left\{\int f\,\ud\mu - \int Q_2f\,\ud\nu\right\}=\frac{1}{2}\overline{\mathcal{T}}_2(\nu|\mu).
\end{align*}

Conversely, if $\tilde g:\R^d\to \R$ is an optimizer in \eqref{eq:dualitybis}, then setting  $\tilde f=P_2\tilde g$, the following inequalities lead to the optimality of the lower semi-continuous convex function  $\tilde f:\R^d\to \R\cup\{+\infty\}$ in \eqref{eq:duality}
\[\frac{1}{2}\overline{\mathcal{T}}_2(\nu|\mu)=\int \tilde g\,\ud\mu - \int P_2\tilde g\,\ud\nu
     \leq \int Q_2\tilde f\,\ud\mu - \int \tilde f\,\ud\nu
     \leq \frac{1}{2}\overline{\mathcal{T}}_2(\nu|\mu).\]
\end{proof}

Note that solutions of the dual problems \eqref{eq:duality} and \eqref{eq:dualitybis} are not unique (even up to constants). In all what follows, we will therefore fix some particular optimizer $\bar{f}$ of \eqref{eq:duality} and set $\bar{g}:=Q_2\bar{f}$. According to the preceding Lemma~\ref{dualitybis}, $\bar{g}$ is an optimizer of \eqref{eq:dualitybis}. Moreover, the convex function $\bar{f}$ can be recovered using the relation $\bar{f}=P_2\bar{g}$ (we refer to \cite[Proposition 2.3]{GRS14} for this classical fact). 

\subsection*{Optimal transport plans and backward projection.} The dual optimizer $\bar{f}$ enables one to describe the optimal transport between $\mu$ and $\nu$. We refer to \cite[Theorem 1.2]{gozlan2020mixture} for all the results stated in this paragraph (see also \cite{BBP19} for alternative proofs). Indeed, let us consider the convex function $\varphi:\R^d \to \R$ defined by 
\begin{equation}\label{eq:varphi}
    \varphi=\left(\bar{f} + \frac{\|\,\cdot\,\|^2}{2}\right)^*.
\end{equation}
This function $\varphi$ is continuously differentiable on $\R^d$ and the map $\nabla \varphi$ is 1-Lipschitz on $\R^d$ . Moreover, the probability measure $\bar{\mu} := \nabla \varphi_\#\mu$ is such that $\bar{\mu} \leq_c \nu$, where $\leq_c$ represents the convex order relation, and is such that 
\begin{equation}\label{eq:GJ20}
\overline{\mc{T}_2}(\nu|\mu)=W_2^2(\mu,\bar{\mu}) = \inf_{\eta \leq_c \nu}W_2^2(\mu,\eta).
\end{equation}
The probability $\bar{\mu}$ is actually the unique minimizer of the right-hand side, and, following the terminology of \cite{Kim-Ruan}, will be called the backward Wasserstein projection of $\mu$ onto the convex set of probability measures dominated by $\nu$ in the convex order, or simply the backward projection of $\mu$ on $\nu$ in the convex order.

Moreover, a transport plan $\pi \in \Pi(\mu,\nu)$ is optimal for $\overline{\mathcal{T}}_2(\nu|\mu)$ if, and only if, it is of the following form
\begin{equation}\label{eq:decomposition}
\ud\pi(x,y) = \ud\mu(x)\ud q_{\nabla \varphi(x)}(y)
\end{equation}
where $q$ is an arbitrary martingale probability kernel transporting $\bar{\mu}$ onto $\nu$, that is $q$ is an arbitrary probability kernel such that $\int q_z(\,\cdot\,)\ud\bar{\mu}(z) = \nu(\,\cdot\,)$ and $\int y\,\ud q_z(y)=z$ for $\bar{\mu}$ almost every $z \in \R^d$. Since $\bar{\mu}\leq_c \nu$, such martingale kernels exist according to the classical characterization of the convex order due to Strassen \cite{Strassen}. We will denote by $\mathcal{M}(\bar{\mu},\nu)$ the set of these martingale kernels. In other words, \eqref{eq:decomposition} means that optimal transport from $\mu$ to $\nu$ are the composition of a deterministic coupling from $\mu$ to $\bar{\mu}$ (along the map $\nabla \varphi$) followed by an arbitrary martingale coupling from $\bar{\mu}$ to $\nu.$

The following lemma will play an important role in Section \ref{sec:paving}.
\begin{lemma}\label{lem:equality-barf}
With the notation introduced above, it holds $\int \bar{f}\,\ud\bar{\mu}=\int \bar{f}\,\ud\nu.$
\end{lemma}
\proof
According to \eqref{eq:duality} and using that $\bar{\mu} \leq_c \nu$, it holds
\begin{align*}
        \frac{1}{2}\overline{\mathcal{T}}_2(\nu |\mu)=\int Q_2\bar{f}\,\ud\mu - \int \bar{f}\,\ud\nu & \leq \int \bar{f}(\nabla \varphi (x)) + \frac{1}{2}\|x-\nabla \varphi(x)\|^2\,\ud\mu(x) - \int f\,\ud\nu\\
        & = \int \bar{f}\,\ud\bar{\mu} + \frac{1}{2}W_2^2(\mu,\bar{\mu}) - \int \bar{f}\,\ud\nu\\   
        & \leq \frac{1}{2}W_2^2(\mu,\bar{\mu})=\frac{1}{2}\overline{\mathcal{T}}_2(\nu |\mu).
    \end{align*}
    Therefore, there is equality everywhere, and so $\int \bar{f}\,\ud\bar{\mu}=\int \bar{f}\,\ud\nu.$
\endproof

\subsection*{Forward projections.}
Following \cite{Alfonsi-Jourdain} and \cite{Kim-Ruan}, we will say that $\tilde{\nu}$ is a forward Wasserstein projection of $\nu$ on the convex cone of measures dominating $\mu$ for the convex order (or simply a forward projection of $\nu$ on $\mu$ in the convex order) if
\[
\tilde{\nu} \in \mathrm{Argmin}_{\mu \leq_c \eta} W_2(\eta,\nu).
\]
As shown in \cite[Theorem 4.1]{Alfonsi-Jourdain} where this notion first appeared, minimizers always exist and are unique if $\nu$ is absolutely continuous with respect to Lebesgue or if $d=1$. Explicit examples where the forward projection is not unique can be found in \cite{Alfonsi-Jourdain-preprint}.  Remarkably, this problem is also related to the quadratic barycentric optimal transport by the following identity first obtained in \cite[Corollary 4.4]{Alfonsi-Jourdain} (see also \cite[Theorem 8.3]{Kim-Ruan}) : 
\begin{equation}\label{eq:AJ}
\overline{\mathcal{T}}_2(\nu|\mu)= W_2^2(\tilde{\nu},\nu),
\end{equation}
for any forward projection $\tilde{\nu}$ of $\nu$ on $\mu$ in the convex order.
Moreover, we have the following result slightly generalizing observations of \cite{Kim-Ruan}.
\begin{lemma}\label{lem:forward-proj}
Let $\tilde{\nu}$ be a forward projection of $\nu$ on $\mu$ in the convex order. Then $\int \bar{g}\,\ud\tilde{\nu} = \int \bar{g}\,\ud\mu$. Moreover, if $\pi$ is an optimal coupling for $W_2^2(\tilde{\nu},\nu)$, then $x\in \partial \varphi^*(y)$ for $\pi$ almost every $(x,y)\in \R^d.$ In particular, if $\nu$ is absolutely continuous with respect to Lebesgue, $\tilde{\nu}$ is the image of $\nu$ under the map $\nabla \varphi^*$ (well defined Lebesgue almost everywhere).
\end{lemma}
For the sake of completeness, we recall the short proof of this result.
\proof
According to \eqref{eq:AJ} and Lemma \ref{dualitybis},
\[
\frac{1}{2}W_2^2(\tilde{\nu},\nu) = \frac{1}{2}\overline{\mathcal{T}}_2(\nu|\mu) = \int \bar{g}\,\ud\mu - \int P_2\bar{g}\,\ud\nu \leq \int \bar{g}\,\ud\tilde{\nu} - \int P_2\bar{g}\,\ud\nu \leq \frac{1}{2}W_2^2(\tilde{\nu},\nu),
\]
and so there is equality everywhere. In particular, the couple $(\bar{g},P_2\bar{g})$ is a dual optimizer for $\frac{1}{2}W_2^2(\tilde{\nu},\nu)$ which classically yields 
\[
  \bar{g}(x)-P_2\bar{g}(y) = \frac{1}{2}\|x-y\|^2
\]
  for $\pi$ almost every $(x,y)\in \R^d.$ Since $\bar{g}=\frac{\|\,\cdot\,\|^2}{2}-\varphi$ and $P_2\bar{g}=\varphi^*-\frac{\|\,\cdot\,\|^2}{2}$, we get that $ \varphi(x)+\varphi^*(y)=x\cdot y$ for $\pi$ almost all $(x,y)\in \R^d$ which completes the proof.
\endproof

\section{Dual potentials, convex paving and martingale couplings}\label{sec:paving}
The aim of this section is to associate to  an optimizer $\bar{f}$ of the dual problem \eqref{eq:duality} a convex paving denoted $\{D_{\bar{f}}(z)\}_{z\in \R^d}$ whose cells are stable by all martingales joining $\bar{\mu}$ to $\nu$, in the following sense : if a martingale $(M_t)_{t\in [0,1]}$ is such that $M_0 \sim \bar{\mu}$ and $M_1\sim\nu$, then, for all $t\in [0,1]$, $M_t\in \overline{D_{\bar{f}}(M_0)}$ almost surely.

\subsection{Convex paving associated to a convex function}
Let $f:\R^d\to \R\cup\{+\infty\}$ be a lower semi-continuous convex function whose domain is denoted $\mathrm {Dom}(f):=\{z\in \R^d:\bar f(z)<+\infty\}$ and is assumed to be non-empty.

For all $z\in \mathrm{Dom}(f)$, consider the convex set
\begin{equation}\label{defD_f}
   D_{f}(z) :=  \left\{y\in \R^d : \exists \varepsilon >0, \forall \lambda\in(-\varepsilon,1+\varepsilon), f\big((1-\lambda)z+\lambda y)\big)= (1-\lambda) f (z) +\lambda  f(y)\right\}. 
\end{equation}

Lemma \ref{lemtopo} below gathers useful properties and alternative characterizations of the family of sets $\{D_{f}(z)\}_{z\in \mathrm{Dom}(f)}$. Before stating this result, one needs to recall some definitions.
\begin{itemize}
\item The closure of a set $A\subset \R^m$ will be denoted by $\overline{A}$.
\item For any set $A \subset \R^m$, the relative face of a point $a\in A$ is defined by 
\begin{align}\label{def-rf_a}
\mathrm{rf}_a(A)&:=\big\{b\in A: \big(a-\varepsilon (b-a),b+\varepsilon (b-a)\big)\subset A \mbox{ for some } \varepsilon>0\big\}\nonumber\\
&=\big\{b\in A: \mbox{ for some } \varepsilon>0, \forall \lambda\in[-\varepsilon, 1+\varepsilon], (1-\lambda)a +\lambda b\in A \big\}
\end{align}
The set $\mathrm{rf}_a(A)$ is equal to the only relative interior of a face of $A$ containing $a$ (see Section 3.1 of \cite{DeMT19}). For any convex set $A$, the family $\{\mathrm{rf}_a(A):a\in A\}$  is known to be a partition of $A$ into convex subsets (see, e.g., \cite[Theorem 18.2]{rockafellar2015convex}).
\item The epigraph of $f$ is the set $E_f \subset \R^d \times \R$ defined by
\begin{equation}\label{defepi}
    E_{ f}:=\{(z,k)\in \R^d\times \R: k\geq f (z)\}.
\end{equation}
The set $E_{f}$ is a closed convex set of $\R^d\times \R$ since $f$ is a lower semi-continuous convex function. 
\item The subgradient of $f$ at the point $z\in  \mathrm{Dom}(f)$ is defined by
\[
\partial f(z):=\{ a\in \R^d:  f(y)\geq  f(z)+a\cdot(y-z)\mbox{ for any } y\in \R^d\}.
\]
\item For any $z\in \mathrm{Dom}(f)$, consider
\begin{equation}\label{defC_f}
    C_{f}(z):=\bigcap_{a\in \partial f(z)}C_f^a(z), \quad \mbox{with} \quad C_f^a(z):=\big\{y\in \R^d : f(y) = f(z)+a\cdot (y-z)\big\}.
\end{equation}
\end{itemize}

With the notation and definitions introduced above, we have the following result.

\begin{lemma}\label{lemtopo} \ 
\begin{enumerate}
\item For any $z\in \mathrm{Dom}( f)$, $C_f(z)$ is a convex subset of $\R^d$ that contains $z$.
We also have 
\[C_f(z)=\big\{y\in \mathrm{Dom}( f): \partial f(z)\subset \partial f(y)\big\},\]
and for any $y\in C_f(z)$,  $C_f(y)\subset C_f(z)$.

\item
    The following assertions are equivalent : for any $y,z \in \mathrm{Dom}( f)$,
    \[\begin{array}{ll}a)\quad y\in\mathrm{proj}_{\R^d}\big(\mathrm{rf}_{(z, {f}(z))} E_f\big),\qquad~ &b)\quad y\in \mathrm{rf}_z C_f(z),\\
    c)\quad  y\in D_f(z), \quad &d) \quad \mathrm{rf}_{(z,{f}(z))} E_f=\mathrm{rf}_{(y, {f}(y))} E_f,
    \end{array}\]
     where $\mathrm{proj}_{\R^d}(y,k):=y$ for any $(y,k)\in \R^d\times \R$. 
    \item For any $z\in \mathrm{Dom}(  f)$, 
    \[D_f(z)=\mathrm{proj}_{\R^d}\big(\mathrm{rf}_{(z,{f}(z))} E_f\big)=\mathrm{rf}_z C_f(z).  \]
    \item The family of sets $\{D_f(z) : z\in \mathrm{Dom}(  f)\}$ forms a partition of $\mathrm{Dom}( f)$ into convex subsets.
    \item For any $z\in \mathrm{Dom}( f)$ and $y\in D_f(z)$, $\partial  f(y)=\partial f(z)$. 
\end{enumerate}
\end{lemma}
The third item helps in particular to visualize how the sets $D_f(z)$ are obtained: this family of sets corresponds to the projection on $\R^d$ of the relative faces of the convex set $E_f\subset \R^d\times \R
$ contained in the graph of $f$. The fact that $\{D_f(z)\}_{z\in \mathrm{Dom}(f)}$ is a partition of $\mathrm{Dom}(f)$ is in particular a direct consequence. The alternative description of $D_f(z)$ involving $C_f(z)$ will be useful to show that the paving $\{D_{\bar f}(z) : z\in \mathrm{Dom}(  \bar f)\}$ is left stable (up to closure) by all martingales interpolating between $\bar \mu$ and $\nu$.

\begin{proof}[Proof of Lemma \ref{lemtopo}] We start with the proof of statement $(i)$. For any $z\in \mathrm{Dom}(  f)$, the subset $C_f(z)$ is  convex and contains $z$  as  intersection of the convex sets $C_f^a(z)$, $a\in \partial  f (z)$, that contain $z$.  Indeed, for any $y,y'\in C_f^a(z)$ and any $\lambda\in[0,1]$, one has 
\[  f\big((1-\lambda)y+\lambda y'\big)\leq (1-\lambda)  f(y)+\lambda f(y')=  f(z)+a\cdot ((1-\lambda)y+\lambda y'-z)\leq  f\big((1-\lambda)y+\lambda y'\big),\]
where the first inequality follows from the convexity property of $ f$ and the second inequality is due to the definition of $\partial f(z)$. This implies  that $(1-\lambda)y+\lambda z\in C_f^a(z)$. 
For the second part of statement $(i)$, for clarity let us denote
\[\widetilde C_f(z):=\big\{y\in \mathrm{Dom}( f): \partial f(z)\subset \partial f(y)\big\},\qquad z\in \mathrm{Dom}(  f).\]
In order to prove that $C_f(z)\subset \widetilde C_f(z)$, let $y\in C_f(z)$ and $a\in \partial f(z)$. We want to  show that $a \in \partial f(y)$. Since $y\in C_f(z)$, it holds $f(y)=f(z)+a\cdot (y-z)$. So, for all $v\in \R^d$, we have
\[
f(v) \geq f(z) + a\cdot (v-z) = f(z)+a\cdot (y-z) + a\cdot (v-y) = f(y) + a\cdot (v-y)
\]
which shows that $a$ belongs to $\partial{f}(y)$.
Conversely, let $y\in \widetilde C_f(z)$. Since $\partial f(z)\subset \partial f(y)$, according to the definition of $\partial f(y)$, one has for all $a\in\partial f(z)\subset \partial f(y)$,
\[f(z)\geq f(y)+a\cdot (z-y).\]
But the converse inequality also holds from the definition of $\partial f(z)$. Therefore equality holds for any $a\in \partial f(z)$, which means that $y\in C_f(z)$. 
The last point of statement $(i)$ is an easy consequence of the characterization of $C_f(z)$ given by  $\widetilde C_f(z)$.

We now turn to the proof of statement $(ii)$.  First,
$d)\implies a)$ is obvious since according to $d)$, one has $(y, f(y))\in \mathrm{rf}_{(z,{f}(z))} E_f$. Let us show that $a)\implies b)\implies c)\implies d)$.

If $a)$ holds, then  there exists $k\in \R$ such that $(y,k)\in \mathrm{rf}_{(z,{f}(z))} E_f$. According to the definition \eqref{def-rf_a}, for some $\varepsilon >0$ and  for all $\lambda \in[-\varepsilon, 1+\varepsilon]$, $(1-\lambda)(z, f(z))+\lambda(y,k)\in E$.
As a consequence, if $a \in \partial f(z)$, then one has for all $\lambda \in[-\varepsilon, 1+\varepsilon]$,
\begin{equation}\label{transit1}(1-\lambda)  f(z)+\lambda k\geq f\big((1-\lambda)z+\lambda y\big)\ge  f(z)+\lambda a \cdot(y-z).
\end{equation}
It follows that for some positive and some negative values of $\lambda$, $\lambda(k- f(z))\geq \lambda a\cdot(y-z)$. Therefore for any $a\in \partial  f (z)$ 
\[ f(y)- f(z)\leq k- f(z)=a\cdot(y-z)\leq  f(y)- f(z),\]
where the first inequality holds since $(y,k)\in E_f$, and the second inequality is due to the definition of $\partial f(z)$. 
Thus we get  $k= f(y)=f(z)+a\cdot (y-z)$ for all $a\in\partial f(z)$. Plugging these equalities  in the left-hand side of \eqref{transit1} provides  for all $\lambda \in [-\varepsilon, 1+\varepsilon]$
\[ f(z)+\lambda a\cdot (y-z)\geq (1-\lambda) f(z)+\lambda  f(y)\geq  f\big((1-\lambda)z+\lambda y\big)\ge f(z)+\lambda a\cdot (y-z),\]
and therefore 
\begin{equation}\label{transit3}
 f\big((1-\lambda)z+\lambda y\big)=  f(z)+\lambda a\cdot (y-z), \quad \forall \lambda \in [-\varepsilon, 1+\varepsilon], \quad \forall a\in \partial  f(z).
\end{equation}
These last property is exactly  $b)$, namely it means $y\in \mathrm{rf}_z C_f(z)$.

Starting  from $b)$, for $\lambda =1$ \eqref{transit3} gives $ f(y)=f(z)+a\cdot (y-z)$ for some $a\in \partial  f (z)$ and therefore also 
\[ f\big((1-\lambda)z+\lambda y\big)= (1-\lambda) f(z)+
\lambda  f(y), \quad \forall \lambda \in [-\varepsilon, 1+\varepsilon],\]
This means that $c)$ holds, namely $y\in D_f(z)$.

Observing that $c)$ obviously implies $(y, f(y))\in \mathrm{rf}_{(z,{f}(z))} E_f$, we get  $d)$, namely  
$\mathrm{rf}_{(z,{f}(z))} E_f=\mathrm{rf}_{(y,{f}(y))} E_f$, since the family $\{\mathrm{rf}_{(z,k)} E_f:(z,k)\in E_f\} $ is a partition of the boundary of the closed convex set $E_f$ (see e.g. \cite[Theorem 18.2]{rockafellar2015convex}). This ends the proof of statement $(ii)$.

Statement $(iii)$ is a consequence of statement $(ii)$, since $d)$ also implies $a)$ with an interchange of the role of $z$ and $y$. 

Let us prove statement $(iv)$. According to statement $(iii)$, the convexity of the set $D_f( z)$ easily follows either from the convexity property of $\mathrm{rf}_{(z,{f}(z))} E_f$ or the one of $\mathrm{rf}_{z} C_f(z)$.  Since $z\in D_f(z)$ for any $z\in \mathrm{Dom}(  f)$, we also get that the family of sets $\{D_f(z):z\in \mathrm{Dom}(  f)\}$  is a partition of $\R^d$ observing that for $z,z'\in\mathrm{Dom}(f)$, $y\in D_f(z)\cap D_f(z')$ implies $\mathrm{rf}_{(z,{f}(z))} E_f=\mathrm{rf}_{(y,{f}(y))} E_f=\mathrm{rf}_{(z',{f}(z'))} E_f$ (according to statement $(ii)$). Therefore statement $(iii)$ gives  $D_f(z)=D_f(z')$.

Let us finally prove statement $(v)$. If $y\in D_f(z)$ then statements $(iii)$ and  $(iv)$ ensure that  $z,y\in D_f(z)=D_f(y)\subset C_f(z)\cap C_f(y)$. It follows from statement $(i)$ that  $\partial f (y)=\partial  f (z)$. This ends the proof of Lemma~\ref{lemtopo}.
\end{proof}

\subsection{Application to martingale-invariant paving}
The construction performed in the preceding subsection can be used to obtain martingale invariant convex paving, as shows the following general result.

\begin{proposition}\label{prop:partitionbis}  
Let $\alpha,\beta$ be two probability measures on $\R^d$ with a finite first moment and such that $\alpha \leq_c \beta$. Suppose that $\int f\,\ud\alpha = \int f\,\ud \beta \in \R$ for some convex lower semicontinuous function $f:\R^d \to \R \cup \{+\infty\}$, then the following properties hold.
\begin{enumerate}
    \item For all $q \in \mathcal{M}(\alpha,\beta)$, $q_z\big(\overline{D_{ f}(z)}\big)=1$ for $\alpha$ almost all $z \in \mathrm{Dom}(f)$.
    \item If $(M_t)_{t\in [0,1]}$ is a martingale with a.s. right continuous trajectories such that $M_0 \sim \alpha$ and $M_1 \sim \beta$, then with probability $1$, for all $t\in [0,1]$, $M_t \in \overline{D_{f}(M_0)}.$
    \item For any $z\in  \mathrm{Dom}( f)$, 
    \[
    \beta\big(\overline{D_{ f}(z)}\big)=\alpha\big(\overline{D_{ f}(z)}\big)+\int_{\R^d\setminus \overline{D_{f}(z)}} q_y\big(\overline{D_{ f}(z)}\cap\big(\overline{D_{f}(y)}\setminus D_{ f}(y)\big) \big)\, \ud\alpha(y).
    \]
    Moreover, if $\beta$ is absolutely continuous, then for any $z\in  \mathrm{Dom}( f)$, $\displaystyle\beta\big(\overline{D_{ f}(z)}\big)=\alpha\big(\overline{D_{ f}(z)}\big)$. 
\end{enumerate}
\end{proposition}

\begin{proof}[Proof of Proposition \ref{prop:partitionbis}] 
Let us prove the first property. Let $q\in \mathcal{M}(\alpha,\beta)$. 
Applying Jensen inequality, we get
\[
\int f\,\ud\beta = \int \left(\int f(y)\,\ud q_z(y)\right)\,\ud\alpha(z) \geq \int  f\left(\int y \,\ud q_z(y) \right)\,\ud\alpha(z) = \int f(z)\,\ud\alpha(z) = \int f\,\ud\beta.
\]
So in particular, for $\alpha$ almost all $z\in \R^d$, we get
 \[
 f(z)=f\left(\int y \,\ud q_z(y)\right) = \int f( y )\,\ud q_z(y).
 \]
 Let $z\in \mathrm{Dom}( f)$ be any such point.  If $a \in \partial  f(z)$, then 
 \[
 \int f( y )\,\ud q_z(y) \geq \int f( z )+a\cdot (z-y)\,\ud q_z(y) = f(z).
 \]
 Thus, for $q_z$ almost all $y$, $f( y )=f( z )+a\cdot (z-y)$. In other words, $q_z(C_f^a(z))=1$. If $(a_i)_{i \geq 1}$ is a dense sequence in $\partial f(z)$, then $C_{f}(z)=\cap_{i\geq 1} C_f^{a_i}(z)$ and therefore  
 \[
 q_z\big(\R^d \setminus C_{ f}(z)\big) \leq \sum_{i\geq 1} q_z\big(\R^d\setminus C_f^{a_i}(z)\big)=0,
 \] 
 and so $q_z( C_{ f}(z))=1$. 
 Since $ C_{f}(z)$ is convex, the subsets $U:=\overline{\mathrm{rf}_z C_{ f}(z)}=\overline{D_{f}(z)}$ and its complement in $C_{ f}(z)$, $V:=C_{ f}(z)\setminus \overline{\mathrm{rf}_z C_{ f}(z)}$, are known to be convex (see e.g \cite[Proposition 3.1]{DeMT19}). It follows that $q_z(V)<1$ since otherwise by the martingale property 
\[
z=\int_{C_{ f}(z)} y \,\ud q_z(y)=\int_{V} y \,\ud q_z(y)\in V,
\] 
which contradicts $z\in \overline{D_{ f}(z)}=U$. Actually, $q_z(V)=0$ and therefore as expected $q_z(U)=1$ since otherwise,   
 \[
 y_0=\frac1{q_z(U)} \int_U  y \,\ud q_z(y)\in U, \quad y_1=\frac1{q_z(V)} \int_V  y \,\ud q_z(y)\in V, 
 \]
 with  $q_z(U)y_0+q_z(V)y_1=z\in U$. This contradicts the fact that according to \cite[Proposition 3.1]{DeMT19}, if $y_1\in V$ and $y_0\in U$ then $[y_1,y_0)\subset V$.

 Let us now prove the second property.
 For $t\in [0,1]$, denote by $\alpha_t$ the law of $M_t$. Since $(M_t)_{t\in [0,1]}$ is a martingale, it holds $\alpha \leq_c\alpha_t\leq_c\beta$. Therefore, 
 \[
 \int f\,\ud \alpha \leq \int f\,\ud \alpha_t \leq \int f\,\ud \beta = \int f\,\ud \alpha,
 \]
 and so $\int f\,\ud \alpha = \int f\,\ud \alpha_t$. Let $q^t=(q^t_z)_{z\in \R^d}$ be the conditional law of $M_t$ knowing $M_0$. Then $q^t\in \mc{M}(\alpha,\alpha_t)$ and applying the first property, we get $q^t_z(\overline{D_f(z)})=1$ for $\alpha$ almost all $z\in \mathrm{Dom}(f)$. Integrating with respect to $\alpha$, yields $\mathbb{P}(M_t \in \overline{D_f(M_0)})=1$. By countable intersection of almost sure events, we get $\mathbb{P}(\bigcap_{t\in \mathbb{Q}\cap [0,1]}\{M_t \in \overline{D_f(M_0)} \})=1$, and finally, by right continuity of the trajectories, $\mathbb{P}(\{\forall t\in [0,1], M_t \in \overline{D_f(M_0)} \})=1$.

 The third statement of Proposition \ref{prop:partitionbis} is an easy consequence of the first one. Since $q$ is a martingale probability kernel transporting $\alpha$ onto $\beta$, the first statement ensures that
 \begin{multline*}
 \beta\big(\overline{D_{f}(z)}\big)=\int q_y\big(\overline{D_{f}(z)}\big)\, d\alpha(y)= \int_{\overline{D_{f}(z)}} q_y\big(\overline{D_{f}(z)}\cap\overline{D_{f}(y)}\big)\, d\alpha(y)+\int_{\R^d\setminus \overline{D_{f}(z)}} q_y\big(\overline{D_{f}(z)}\cap\overline{D_{f}(y)}\big)\, d\alpha(y).\end{multline*}
  If $y\in \overline{D_{f}(z)}$ then $y\in C_{f}(z)$ since according  to Lemma \ref{lemtopo}, $\overline{D_{f}(z)}=\overline{\mathrm{rf}_z C_{f}(z)}\subset \overline{ C_{f}(z)}=C_{f}(z)$. As a consequence one has $C_{f}(y)\subset C_{f}(z)$ and therefore $\overline{D_{f}(y)}=\overline{\mathrm{rf}_y C_{f}(y)}\subset \overline{\mathrm{rf}_z C_{f}(z)}= \overline{D_{f}(z)}$. It follows that $q_y\big(\overline{D_{f}(z)}\cap\overline{D_{f}(y)}\big)=q_y\big(\overline{D_{f}(y)}\big)=1$. Conversely, if  $y\notin \overline{D_{f}(z)}$ then necessarily  $D_{f}(y)\cap \overline{D_{f}(z)}=\emptyset$ (otherwise, for some $y'\in D_{f}(y)\cap \overline{D_{f}(z)}$,  same arguments as above imply  $y\in D_{f}(y)=D_{f}(y')\subset \overline{D_{f}(y')}\subset \overline{D_{f}(z)}$). These last observations lead to 
  \[
  \beta\big(\overline{D_{ f}(z)}\big)=\alpha\big(\overline{D_{ f}(z)}\big)+\int_{\R^d\setminus \overline{D_{f}(z)}} q_y\big(\overline{D_{ f}(z)}\cap\big(\overline{D_{f}(y)}\setminus D_{ f}(y)\big) \big)\, \ud\alpha(y).
  \]
  Now, let us assume that $\beta$ is absolutely continuous with respect to Lebesgue and let $z\in \mathrm{Dom}(f)$. If $D_f(z)$ has dimension less than or equal $d-1$, then $\beta(D_f(z))=0$, and so $\alpha(D_f(z))=0$ also. Let us assume that $D_f(z)$ has dimension $d$. 
  Note that, if $y\in \R^d\setminus \overline{D_{f}(z)}$, then $D_f(y) \subset \R^d\setminus D_{f}(z)$ and so $\overline{D_{ f}(z)}\cap\big(\overline{D_{f}(y)}\setminus D_{ f}(y)\big) \subset \overline{D_{ f}(z)}\cap\overline{D_{f}(y)} \subset \overline{D_{ f}(z)}\cap\overline{\R^d\setminus D_{f}(z)} = \partial D_f(z)$ and so 
\[
\int_{\R^d\setminus \overline{D_{f}(z)}} q_y\big(\overline{D_{ f}(z)}\cap\big(\overline{D_{f}(y)}\setminus D_{ f}(y)\big) \big)\, \ud\alpha(y)\leq \int_{\R^d\setminus \overline{D_{f}(z)}} q_y\big(\partial D_f(z)\big)\, \ud\alpha(y)\leq \int_{\R^d} q_y\big(\partial D_f(z)\big)\, \ud\alpha(y)=\beta(\partial D_f(z))=0,
\]
since $\beta$ is absolutely continuous.
So we have
\[
  \beta\big(\overline{D_{ f}(z)}\big)=\alpha\big(\overline{D_{ f}(z)}\big),
\]
which completes the proof.
\end{proof}

We can now immediately derive from the preceding result the announced martingale-invariant paving.
\begin{corollary}\label{cor:invariant-paving} Let $M=(M_t)_{t\in [0,1]}$ and $N=(N_t)_{t\in [0,1]}$ be two martingales with a.s. right continuous trajectories.
\begin{enumerate}
    \item If $(M_t)_{t\in [0,1]}$ is such that $M_0 \sim \bar{\mu}$ and $M_1 \sim \nu$, then with probability $1$, for all $t\in [0,1]$, $M_t \in \overline{D_{\bar f}(M_0)}.$
    \item If $(N_t)_{t\in [0,1]}$ is such that $N_0 \sim \mu$ and $N_1 \sim \tilde{\nu}$, then with probability $1$, for all $t\in [0,1]$, $N_t \in \overline{D_{\bar g}(N_0)}.$
\end{enumerate}
\end{corollary}
\proof
According to Lemmas \ref{lem:equality-barf} and \ref{lem:forward-proj}, one gets $\int \bar{f}\,\ud \bar{\mu} = \int \bar{f}\, \ud \nu$ and $\int \bar{g}\,\ud \mu = \int \bar{g}\, \ud \tilde{\nu}$. Since $\bar{\mu} \leq_c \nu$ and $\mu \leq_c \tilde{\nu}$, the conclusion follows from Item 2. of Proposition \ref{prop:partitionbis}.
\endproof

To summarize, in what precedes, we observed that given two probability distributions $\mu,\nu \in \mc P_2(\R^d)$, any dual optimizer $\bar{f}$ in \eqref{eq:duality} is such that $\int \bar{f}\,\ud \bar{\mu} = \int \bar{f}\,\ud\nu$, which then yields, thanks to the general Proposition \ref{prop:partitionbis}, the martingale invariant paving $\big\{\overline{D_{ \bar f}(z)}\big\}_{z\in \mathrm{Dom}(\bar{f})}$.
Corollary \ref{cor:invariant-paving} is reminiscent of a deep result of de March and Touzi \cite{DeMT19} showing that for any $\alpha \leq_c \beta$ there exists a convex paving (minimal in some sense) adapted to all martingale kernels sending $\alpha$ onto $\beta$ (see Theorem 2.1 of \cite{DeMT19} for a precise statement). While our result only covers the case where $\alpha = \bar{\mu}$ is the projection of a given probability measure $\mu$ onto the set of all probability measures dominated by $\beta=\nu$ for the convex order, it has the interest of linking in a constructive way the paving with the dual potential $\bar{f}$.

The following proposition explores some converse construction.

\begin{proposition}\label{prop-constr-rec}
Let $\alpha,\nu$ be two compactly supported probability measures on $\R^d$ such that $\alpha \leq_c \nu$ and a  lower semicontinuous convex function $k:\R^d \to \R$ such that $\int k\,\ud \alpha = \int k\,\ud \nu$. There exists a compactly supported probability measure $\mu$ such that $\alpha = \bar{\mu}$ is the backward projection of $\mu$ on $\nu$ in the convex order, and such that $k$ is a dual optimizer for $\overline{\mc T_2}(\nu|\mu)$. 
\end{proposition}
In other words, under mild assumptions on $\alpha\leq_c \nu$, all convex functions saturating the convex order between $\alpha$ and $\nu$ arise as dual optimizers of some quadratic barycentric transport problem realizing $\alpha$ as a backward projection on $\nu.$
\proof
Define $\varphi= \left(k+\frac{\|\,\cdot\,\|^2}{2}\right)^*$ so that $\varphi^* = k+\frac{\|\,\cdot\,\|^2}{2}$. The function $\varphi$ is continuously differentiable as a Fenchel-Legendre transform of a strictly convex function. Since $\varphi^*$ takes finite values on $\R^d$, for all $z\in \R^d$, $\partial \varphi^*(z)$ is a non-empty, compact and convex subset of $\R^d$. Moreover, it is not difficult to check that the diameter of $\partial \varphi^*(z)$ is uniformly bounded for $z$ in the support of $\alpha$, denoted $K_\alpha$ in what follows. 
Consider a kernel $r=(r_z)_{z\in K_\alpha}$ such that, for all $z\in K_\alpha$, $r_z(\partial \varphi^*(z))=1$. One can, for instance, take
for $r_z$ the uniform probability measure on $\partial \varphi^*(z)$ defined as $r_z(\,\cdot\,)=\frac{\mathcal{H}_\ell\left(\,\cdot\, \cap \partial \varphi^*(z)\right)}{\mathcal{H}_\ell(\partial \varphi^*(z))}$, where $\mathcal{H}_\ell$ denotes the $\ell$ dimensional Hausdorff measure with $\ell$ being the dimension of the affine hull of $\partial \varphi^*(z)$). Consider now $\mu = \alpha r$. It is easily seen that $\mu$ has compact support. We claim that $\nabla \varphi_\# \mu = \alpha$. Namely, if $x \in \partial \varphi^*(z)$, then by duality of subgradient, $z\in \partial \varphi(x) = \{\nabla \varphi(x)\}$ and so $\nabla \varphi(x)=z$. Thus, for any bounded test function $a$, 
\[
\int a(\nabla \varphi)\,\ud \mu = \int \left( \int a (\nabla \varphi (x)) \,\ud r_z(x)\right) \,\ud\alpha(z) = \int a(z) \,\ud \alpha(z).
\]
Since $\varphi$ is convex, $\nabla \varphi$ is an optimal transport map for $W_2^2(\mu,\alpha)$. Integrating $x\cdot \nabla \varphi(x) = \varphi(x)+\varphi^*(\nabla \varphi(x))$ with respect to $\ud \mu(x)$ yields
\[
\int \varphi\,\ud\mu + \int \varphi^*\,\ud\alpha = \int x\cdot \nabla \varphi(x)\,\ud \mu(x)
\]
from which, by subtracting $1/2$ times the second order moments of $\mu$ and $\alpha$, it follows that
\[
\frac{1}{2}W_2^2(\mu,\alpha) = \int Q_2k\,\ud\mu-\int k\,\ud\alpha.
\]
By assumption, $\int k\,\ud\alpha = \int k\,\ud\nu$ and $\alpha \leq_c \nu$, so
\[
\frac{1}{2}W_2^2(\mu,\alpha) = \int Q_2k\,\ud\mu-\int k\,\ud\nu \leq \frac{1}{2}\mc T_2(\nu|\mu) = \inf_{\eta \leq_c \nu} \frac{1}{2}W_2^2(\mu,\eta) \leq \frac{1}{2}W_2^2(\mu,\alpha).
\]
So, we conclude that, 
\[
W_2^2(\mu,\alpha) = \inf_{\eta \leq_c \nu} W_2^2(\mu,\eta)=\int Q_2k\,\ud\mu-\int k\,\ud\nu,
\]
and so $\alpha$ is the backward projection of $\mu$ on $\nu$ in the convex order, and $k$ is a dual optimizer.
\endproof

\subsection{Construction of the martingale-invariant paving from the optimal potential \texorpdfstring{$\bar g=Q_2\bar f$}{g=Q2fw}. }

According to Lemma \ref{dualitybis}, if $\bar{f}$ is an optimal potential in the dual equality \eqref{eq:duality}, then $\bar g=Q_2\bar f$ is also an optimizer of the dual equality \eqref{eq:dualitybis}. Moreover, $\bar f$ and $\bar g$ are also related by the following properties.

\begin{proposition}\label{prop-lienfg}
     According to the notation of the last section, $\bar g=Q_2 \bar f=\frac{\|\cdot\|^2}2-\varphi $ is continuously differentiable in $\R^d$ with $\nabla \bar g(x)=x-\nabla\varphi(x)$, $x\in \R^d$. Recall that $\psi=\varphi^*=\bar f+\frac{\|\cdot\|^2}2$. One has 
\begin{enumerate}
 \item For any $z\in\mathrm{Dom}( \bar f)$ and any $a\in \partial \bar f (z)$, \[C_f^a(z)=\{y\in\R^d : \nabla \bar g(a+z)=\nabla \bar g(a+y)\},\]
 where the definition of the set $C_f^a(z)$ is given by \eqref{defC_f}.
 \item For any $z\in \mathrm{Dom}( \bar f)$, $\partial \bar f(z)=\{\nabla \bar g(x): x\in (\nabla \varphi)^{-1} (z)\}$.
 \end{enumerate} 
\end{proposition}

\begin{proof}
    Let us prove Item $(i)$. For $z\in \mathrm{Dom}( \bar f)$, let $a\in \partial f(z)$. For clarity, let us first denote \[\widetilde C_f^a(z):=\{y\in\R^d : \nabla \bar g(a+y)=\nabla \bar g(a+z)\}.\] 
    As a useful tool for this proof, observe that since $\psi=\bar f+\frac{\|\cdot \|^2}2$ and $\psi^*=\varphi =\frac{\|\cdot \|^2}2-\bar g$, one has 
    \begin{equation}\label{equiv}a\in \partial\bar f(z)\Leftrightarrow a+z\in \partial \psi (z) \Leftrightarrow \nabla\varphi (a+z)=z\Leftrightarrow  \nabla \bar g(a+z)=a.
    \end{equation}
    Assume first that $y\in \widetilde C_f^a(z)$. From the last observation,  $a=\nabla \bar g(a+z)=\nabla \bar g(a+y)$ and therefore $a\in \partial \bar f(y)$. Since $a\in \partial \bar f(z)$, one has $\bar f(y)\geq \bar f(z)+a\cdot(y-z)$ and since $a\in\partial \bar f(z)$, one has $\bar f(z)\geq \bar f(y)+a\cdot(z-y)$, which implies that $y\in C_f^a(z)$.
Conversely, if $y\in C_f^a(z)$  then by definition of $C_f^a(z)$, it holds 
\[
\bar{f}(y) = \bar{f}(z)+a\cdot (y-z)
\]
and so 
\[
a\cdot y - \bar{f}(y) = a\cdot z - \bar{f}(z) = \bar{f}^*(a),
\]
where the second equality comes from $a \in \partial \bar{f}(z)$. The equality 
$
\bar{f}(y) = a\cdot y - \bar{f}^*(a)
$
 implies that $a \in \partial f(y)$.
According to \eqref{equiv}, it follows that  $\nabla \bar{g}(a+y)=a$. But, since $a \in \partial \bar{f}(z)$, we also have $\nabla \bar{g}(a+z)=a$, and so $\nabla \bar{g}(a+z)=\nabla \bar{g}(a+y)$ and $y\in \widetilde{C}^a(z).$ This ends the proof of the second part of Proposition \ref{prop-lienfg}.

 Item $(ii)$ is also a consequence of \eqref{equiv}.  On one side it gives that if $a\in \partial \bar f(z)$ then $a=\nabla \bar g(x)$ with $x=a+z\in (\nabla \varphi)^{-1}(z)$. Conversely if $a=\nabla \bar g(x)$ with $x\in (\nabla \varphi)^{-1}(z)$ then $a\in \partial \bar f (x-a)$ with $x-a=z$. The proof of  Proposition  \ref{prop-lienfg} is therefore completed.
\end{proof}

\section{Benamou-Brenier type formula for \texorpdfstring{$\overline{\mathcal{T}}_2$}{T2}}\label{sec:benamou}
This section develops a notion of geodesic on the space $\mc{P}_2(\R^d)$ equipped with the cost functional $\overline{\mathcal{T}}_2$.
If $x=(x_t)_{t\in [0,1]}$ is a continuous path taking values in some metric space $(E,d)$, the length of $x$ is usually defined as follows:
\[
\mathrm{Length}((x_t)_{t\in [0,1]}) := \sup_{0=t_0<t_1<\cdots<t_n=1}\sum_{i=0}^{n-1}d(x_{t_i},x_{t_{i+1}}).
\]
Even if $\overline{\mathcal{T}}_2^{1/2}(\,\cdot\,|\,\cdot\,)$ is not a distance, we will define the oriented length of a continuous path $(\mu_t)_{t\in[0,1]}$ interpolating between $\mu_0=\mu$ and $\mu_1=\nu$ as follows:
\[
\ell((\mu_t)_{t\in [0,1]}):=\sup_{0=t_0<t_1<\cdots<t_n=1}\sum_{i=0}^{n-1}\sqrt{\overline{\mc{T}}_2(\mu_{t_{i+1}}|\mu_{t_i})}.
\]
The following result shows that the space $\big(\mc{P}_2(\R^d),\overline{\mathcal{T}}^{1/2}_2\big)$ is a sort of geodesic space, in the sense that given two probability measures $\mu,\nu$ the ``distance" $\sqrt{\overline{\mc{T}}_2(\nu|\mu)}$ coincides with the shortest length of a path that joins $\mu$ to $\nu$. This result is inspired by the classical Benamou-Brenier formula for the standard quadratic Wasserstein distance $W_2$ on $\mc{P}_2(\R^d)$ \cite{Benamou-Brenier}.

In the following, we fix a filtered probability space $(\Omega,\mathcal{F},\mathbb{P})$ sufficiently rich to support the existence of a standard $d$ dimensional Brownian motion $B$ and an $\mathcal{F}_0$-measurable random vector $X_0$ with law $\mu$.

\newpage
\begin{theorem}\label{thm:BBbar}
Let $\mu,\nu\in\mc{P}_2(\R^d)$.
\begin{enumerate}
    \item It holds  
\begin{equation}\label{eq:BBbar}
\overline{\mc{T}}_2(\nu|\mu)=\inf\mathbb{E}\int_0^1\|v_t\|^2\ud t,
\end{equation}
where the infimum is taken over all progressively measurable processes  $(v_t)_{t\in [0,1]}$ with the following constraints: there exists an $\mathcal{F}_0$-measurable random variable $X_0\sim \mu$ and an $\mc{F}$-martingale $(M_t)_{t\in [0,1]}$ such that the process $X_t = X_0+\int_0^t v_s\,ds +M_t-M_0$, $t\in [0,1]$, is such that $X_1$ has law $\nu$. 
\item Moreover
\begin{equation}\label{eq:geod}
\overline{\mc{T}}_2^{1/2}(\nu|\mu)=\inf \ell((\mu_t)_{t\in [0,1]}),
\end{equation}
where the infimum runs over the set of paths $(\mu_t)_{t\in [0,1]}$ such that $\mu_0=\mu$ to $\mu_1=\nu$.
\item Equality in \eqref{eq:BBbar} is achieved for any  process of the form
\begin{equation}\label{eq:Xopt}
X_t = X_0+t(\nabla \varphi(X_0)-X_0) + M_t-\nabla\varphi(X_0),\qquad t\in [0,1]
\end{equation}
(which corresponds to $v_t=\nabla \varphi(X_0)-X_0$, $t\in [0,1]$)
with 
\begin{itemize}
    \item $\varphi$ is a continuously differentiable convex function such that $\nabla\varphi$ pushes forward $\mu$ onto $\overline{\mu}:= \mathrm{argmin}_{\eta \leq_c \nu} W_2(\mu,\eta)$,
    \item $(M_t)_{t\in [0,1]}$ is an $\mathcal{F}$-martingale with a.s. continuous trajectories such that $M_0=\nabla\varphi(X_0)$ a.s. and $M_1\sim \nu$.
\end{itemize}
In this case, $\mu_t=\mathrm{Law}(X_t)$, $t\in [0,1]$, achieves equality in \eqref{eq:geod} and it holds
\begin{equation}\label{eq:geodbar}
\overline{\mc{T}}_2^{1/2}(\mu_t|\mu_s) = (t-s)\overline{\mc{T}}_2^{1/2}(\nu|\mu),\qquad \forall 0\leq s\leq t \leq 1.
\end{equation}
\item Let $X$ be as in Item $(iii)$.  If $M$ is Markovian, then so is $X$.
\end{enumerate}
\end{theorem}
The proof will in particular establish the existence of a time inhomogenous Markov process $X$ achieving equalities in \eqref{eq:BBbar} and \eqref{eq:geod}.
\begin{proof}
Let $\ud X_t = v_t\,\ud t+\ud M_t$ be admissible in \eqref{eq:BBbar}, denote by $\mu_t=\mathrm{Law}(X_t)$ and let $0=t_0<t_1<\ldots<t_n=1$ be a subdivision of $[0,1]$.
By definition of $\overline{\mathcal{T}}_2(\mu_{t_{i+1}} | \mu_{t_i})$, 
\[
\overline{\mathcal{T}}_2(\mu_{t_{i+1}} | \mu_{t_i}) \leq \mathbb{E}\left[\|\mathbb{E}[X_{t_{i+1}} - X_{t_i}|X_{t_i}]\|^2\right]
\]
Since $M$ is a martingale, we get
\[
\mathbb{E}[X_{t_{i+1}} - X_{t_i}|X_{t_i}] = \mathbb{E}[\int_{t_i}^{t_{i+1}} v_s\,\ud s|X_{t_i}].
\]
So, applying two times Jensen's inequality,
\[
\overline{\mathcal{T}}_2(\mu_{t_{i+1}} | \mu_{t_i}) \leq \mathbb{E}\left[\|\int_{t_i}^{t_{i+1}} v_s\,\ud s\|^2\right] \leq (t_{i+1}-t_i)\mathbb{E}\left[\int_{t_i}^{t_{i+1}} \|v_s\|^2\,\ud s\right].
\]
Hence, by Cauchy-Schwarz,
\begin{align*}
\sum_{i=0}^{n-1} \overline{\mathcal{T}}_2^{1/2}(\mu_{t_{i+1}} | \mu_{t_i}) &\leq \sum_{i=0}^{n-1}(t_{i+1}-t_i)^{1/2}\mathbb{E}\left[\int_{t_i}^{t_{i+1}} \|v_s\|^2\,\ud s\right]^{1/2}\\
& \leq \left(\sum_{i=0}^{n-1}t_{i+1}-t_i\right)^{1/2}\left(\sum_{i=0}^{n-1}\mathbb{E}\left[\int_{t_i}^{t_{i+1}} \|v_s\|^2\,\ud s\right]\right)^{1/2}\\
& = \left(\mathbb{E}\int_0^1\|v_s\|^2\,\ud s\right)^{1/2}.
\end{align*}
By taking the supremum over all partitions, we conclude that
\begin{equation}\label{eq:ineqBB}
\overline{\mathcal{T}}_2^{1/2}(\nu | \mu) \leq \ell((\mu_t)_{t\in [0,1]}) \leq \left(\mathbb{E}\int_0^1\|v_s\|^2\,\ud s\right)^{1/2}.
\end{equation}
Now, assume that $X$ is given by \eqref{eq:Xopt}. Then $v_s=\nabla \varphi (X_0) - X_0$ for all $s\in [0,1].$ Thus 
\[
\mathbb{E}\int_0^1\|v_s\|^2\,\ud s = \mathbb{E}\int_0^1\|\nabla \varphi (X_0) - X_0\|^2\,\ud s = W_2^2(\mu,\overline{\mu}) = \overline{\mathcal{T}}_2(\nu | \mu),
\]
so all inequalities in \eqref{eq:ineqBB} are saturated, proving $(i)$, $(ii)$ and $(iii)$ and \eqref{eq:geodbar}.
It remains to prove the existence of a process $X$ as in Item $(iii)$. In order to construct the martingale part, we will rely on the notion of stretched Brownian motion introduced in \cite{backhoff2020martingale}. 
Consider a standard Brownian motion $B=(B_t)_{t\in [0,1]}$ adapted to the filtration $\mathcal{F}$ and let $X_0$ be an $\mathcal{F}_0$-measurable random variable with distribution $\mu$ independent of $B$.
Consider the following weak optimal transport problem:
\[
\inf_{p} \int  W_2^2(p_x,\gamma)\,\ud\overline{\mu}(x),
\]
where $\gamma$ is the standard Gaussian distribution on $\mathbb{R}^d$ and the infimum runs over the set of probability kernels $p=(p_x)_{x\in \mathbb{R}^d}$ such that $\int y\,\ud p_x(y)=x$ for $\mu$ almost every $x\in \mathbb{R}^d$ and $\overline{\mu}p=\nu$. Since $\overline{\mu} \leq_c \nu$ this set of kernels is non-empty. According to \cite[Theorem 2.2]{backhoff2020martingale}, there exists a $\overline{\mu}$ a.s. unique minimizer $p^*$. For all $x\in \mathbb{R}^d$, let $F^x$ be a convex function such that $\nabla F^x$ pushes forward $\gamma$ to $p^*_x$ and set
\[
g_t(x,b) = \int \nabla F^x(y+b)\,\ud\gamma_{1-t}(y),
\]
where $\gamma_{1-t}$ is the centered Gaussian distribution with covariance $(1-t)I_d$. Setting 
\[
M_t = g_t(\nabla \varphi(X_0), B_t),\qquad t\in [0,1]
\]
it is easily seen that $M_0=\nabla\varphi(X_0)$ a.s. and $M_1 \sim \nu$. Moreover, since for all $x$, $g_t(x,B_t) = \mathbb{E}[\nabla F^x(B_1)|\mathcal{F}_t]$, one concludes that $M$ is an $\mathcal{F}$-martingale. It follows from Corollary 2.5 of \cite{backhoff2020martingale}, that $M$ is Markovian.

Finally, let us prove Item $(iv)$.
Let us first construct a function $w_t$ such that $w_t (X_t) = X_0 - \nabla \varphi (X_0)$ almost surely, that is
\[
w_t((1-t) (X_0-\nabla \varphi (X_0)) + M_t) = X_0-\nabla \varphi (X_0)
\]
almost surely.
We recall that 
\[
\varphi = \left(\bar{f} + \frac{1}{2}|\,\cdot\,|^2\right)^*.
\]
By duality of the subgradients, we have
\[
\nabla \varphi(x) \in \partial \varphi(x) \Rightarrow x \in \partial \varphi^*(\nabla \varphi(x)) = \partial \bar{f}(\nabla \varphi(x)) + \{\nabla \varphi (x)\}.
\]
Therefore, for all $x \in \R^d$,
\[
x-\nabla \varphi(x) \in \partial \bar{f}(\nabla \varphi(x)).
\]
So, almost surely, $X_0-\nabla \varphi(X_0) \in \partial \bar{f}(\nabla \varphi(X_0))$.

Since $M$ has a.s. continuous trajectories, applying Corollary \ref{cor:invariant-paving} yields that, with probability $1$, for all $t \in [0,1]$, $M_t \in \overline{D_{\bar{f}}(\nabla \varphi(X_0))}$.

Hence, it suffices to construct a function $w_t$ such that, 
\[
w_t((1-t) u + m)=u
\]
for all $u \in \partial \bar{f}(\nabla \varphi(x_0))$ and all $m \in \overline{D_{\bar{f}}(\nabla\varphi (x_0))}$ and all $x_0.$
According to Item $(iii)$ of Lemma \ref{lemtopo}, $\overline{D_{\bar{f}}(\nabla\varphi (x_0))} \subset C_{\bar{f}}(\nabla\varphi (x_0)))$. Thus, according to Item $(i)$ of Lemma \ref{lemtopo}, one has $\partial \bar{f}(\nabla \varphi(x_0)) \subset \partial \bar f (m)$.
Thus, it is enough to construct a function $w_t$ such that
\[
w_t((1-t) u + m)=u
\]
for all $u \in \partial \bar{f}(m)$ and for all $m \in \R^d$.
Let $\partial \bar{f} = \{(m,u) : m \in \R^d, u \in \partial \bar{f}(m)\}$ be the subgradient of $\bar{f}$.
Consider the map 
\[
H_t : \partial \bar{f} \to \R^d : (m,u) = (1-t)  u+ m.
\]
We claim that, for $t\in [0,1)$, $H_t$ is a bijection from $ \partial \bar{f}$ onto $H_t( \partial \bar{f})$.
Indeed, if $H_t((m,u)) = H_t((m',u'))$, then
\[
(1-t)  u+ m = (1-t) u'+ m'
\]
and so 
\[
(1-t)(u-u') = m'-m
\]
and so, taking the scalar product with $m-m'$ we get
\[
(1-t)(u-u')\cdot (m-m') = -\|m-m'\|^2.
\]
Since the function $\bar{f}$ is convex, the left hand side is $\geq0$, and so $m=m'$, and therefore $u=u'$ since $t\neq 1$. Finally, the function $w_t$ defined by
\[
w_t (x) = p_2 (H_t^{-1}(x)),\qquad \forall x \in H_t( \partial \bar{f}),
\]
with $p_2((m,u)) = u$, $(m,u) \in \R^d \times \R^d$, satisfies all the requirements.

Let us now show that $X$ is a Markov process w.r.t the filtration $\mathcal{F}$, as soon as $M$ is Markov.
We want to prove that, for all $0\leq t<u\leq 1$, it holds
\[
\esp [h(X_{u}) | \mc F_t] = \esp [h(X_{u}) | X_t],\qquad a.s,
\]
for all bounded continuous function $h$. It is enough to consider functions $h$ of the form $h_\lambda(x)=e^{i \lambda \cdot x}$, $x\in \R^d$, $\lambda \in \R^d$.
Note that
\[
X_u = (1-u)(X_0-\nabla \varphi(X_0)) + M_u = (1-u)\omega_t(X_t) + M_u
\]
and so
\[
\esp [h_\lambda(X_{u}) | \mc F_t] = h_{\lambda}((1-u)\omega_t(X_t))\esp [h_\lambda(M_{u}) | \mc F_t]=h_{\lambda}((1-u)\omega_t(X_t))\esp [h_\lambda(M_{u}) |M_t],
\]
using that $M$ is Markov.
Since $M_t=X_t-(1-t)w_t(X_t)$ is $\sigma(X_t)$-measurable, we conclude  that $\esp [h_\lambda(X_{u}) | \mc F_t]$ is $\sigma(X_t)$-measurable. This easily implies that $\esp [h_\lambda(X_{u}) | \mc F_t]=\esp [h_\lambda(X_{u}) |X_t]$ and ends the proof.
\end{proof}

 \begin{remark}\label{rem:alternative}Let $X$ be an optimal process as in Item $(iii)$ of Theorem \ref{thm:BBbar} and let $\mu_s = \mathrm{Law}(X_s)$, $s\in [0,1]$.
 \begin{itemize}
\item It is instructive to analyse the equality cases in the proof of \eqref{eq:geodbar}. For all $0\leq s<t\leq 1$, we get 
 \begin{align*}
 \overline{\mathcal{T}}_2(\mu_t | \mu_s) & \leq \esp[\|\esp [X_t-X_s | X_s]\|^2]\\
 & =  (t-s)^2 \esp[\|\esp [ \nabla \varphi (X_0)-X_0 | X_s]\|^2]\\
 & \leq (t-s)^2 \esp[\esp [\| \nabla \varphi (X_0)-X_0 \|^2 | X_s]]\\
 & = (t-s)^2  \overline{\mathcal{T}}_2^{1/2}(\nu | \mu)\\
 & =  \overline{\mathcal{T}}_2(\mu_t | \mu_s).
 \end{align*}
Thus, by the equality case of Jensen inequality at the third line, we obtain that $X_0-\nabla\varphi(X_0)$ is in fact $X_s$-measurable.
In other words, we recover that, for all $s\in[0,1)$, there exists a measurable function $w_s:\R^d\to\R^d$ such that $X_0-\nabla\varphi(X_0)=w_s(X_s)$ almost surely, a property that was used in the proof of Item $(iv)$ of Theorem \ref{thm:BBbar}.
\item The reasoning above also shows that $(X_s,X_t)$ is an optimal coupling between $\mu_s$ and $\mu_t$. 
\item Let $s\in [0,1]$ and denote by $\overline{\mu_s}$ the backward projection of $\mu_s$ on $\nu(=\mu_1)$ in the convex order. Since $(X_s,X_1)$ is an optimal coupling for $\overline{\mc T_2}(\nu|\mu_s)$, it follows from \eqref{eq:decomposition} that $\esp[X_1 | X_s] \sim \overline{\mu_s}.$
But, 
\[
\esp[X_1 | X_s] = \esp[M_1 | X_s] = \esp[\esp[M_1|\mathcal{F}_s]| X_s] = \esp[M_s|X_s] = M_s,
\]
where we used the fact that $M$ is a martingale such that  $M_s$ is $\sigma(X_s)$-measurable, as shown in the proof of Item $(iv)$ of Theorem \ref{thm:BBbar} or in the first Item of this remark. We thus conclude that $M_s \sim \overline{\mu_s}.$
\end{itemize}
\end{remark}

\section{Martingale transport toward forward projections}\label{sec:forward-proj}
The following proposition is the main result of this section. It establishes a correspondence between martingales joining $\mu$ to one of its forward projections $\tilde{\nu}$ and those joining $\bar{\mu}$ to $\nu.$
\begin{proposition}\label{prop:marti-transfer}
Let $\tilde{\nu}$ be a forward projection of $\mu$ on $\nu$ in the convex order.
If $(N_t)_{t\in [0,1]}$ is a martingale such that $N_0\sim \mu$ and $N_1\sim \tilde{\nu}$ then $M_t = \nabla \varphi(N_t)$, $t\in [0,1]$, is a martingale such that $M_0 \sim \bar{\mu}$ and $M_1 \sim \nu$.
\end{proposition}
Note that the conclusion of this result is surprising, since the image of a martingale under a non-linear transformation is in general not a martingale.
The proof of Proposition \ref{prop:marti-transfer}  relies on the following  simpler description of the sets $C_{\bar{g}}(x)$, $x\in \R^d$.

\begin{proposition}\label{prop:Cbarg}
For all $x\in \R^d$, 
\[
C_{\bar{g}}(x) = \{u \in \R^d : \nabla \varphi(u) = \nabla \varphi(x)+u-x\},
\]
where $\nabla \varphi$ is the transport map between $\mu$ and $\bar{\mu}$.
\end{proposition}
In other words, the pieces of the space that are stable by the martingales transporting $\mu$ onto $\tilde{\nu}$ correspond to the cells where the transport map $\nabla \varphi$ sending $\mu$ onto $\bar{\mu}$ acts as a translation. 
\proof Recall that $\bar{g}=Q_2\bar{f} = \frac{|\,\cdot\,|^2}{2} - \varphi$ is a continuously differentiable convex function defined on $\R^d$, so that for any $x\in \R^d$, 
\[\partial \bar g(x)=\{\nabla \bar g(x)\}=\{x-\nabla\varphi(x)\}.\]
According to Lemma \ref{lemtopo} $(i)$, it follows that 
\[C_{\bar{g}}(x)= \{u \in \R^d : \partial \bar{g}(x)\subset \partial \bar{g}(u)\}
     =\{u \in \R^d : \nabla \varphi(u) = \nabla \varphi(x)+u-x\}.\]
     \endproof

\proof[Proof of Proposition \ref{prop:marti-transfer}.]
It is clear that $M_0=\nabla \varphi(N_0) \sim \bar{\mu}$. Let us show that $M_1=\nabla \varphi(N_1)$ has law $\nu$. This amounts to proving $\nabla \varphi_\#\tilde{\nu}=\nu$. Consider $(U,V)$ an optimal coupling for $W_2^2(\tilde{\nu},\nu)$. Then according to Lemma \ref{lem:forward-proj}, $U \in \partial \varphi^*(V)$ a.s. But, then $V\in \partial \varphi(U) = \{\nabla \varphi(U)\}$ and so $V= \nabla\varphi(U)$ which proves that $\nabla \varphi_\#\tilde{\nu}=\nu$ and $M_1 \sim \nu.$
Finally, let us show that $M$ is also a martingale. Let $0\leq s<t\leq 1$ ; since $N_s,N_t \in C_{\bar{g}}(N_0)$ a.s, we get according to Proposition \ref{prop:Cbarg} that
\[
\nabla \varphi(N_t) = \nabla \varphi(N_s) + N_t-N_s
\]
a.s, and thus $\mathbb{E}[M_t | \mathcal{F}_s] = M_s$ a.s, which completes the proof.
\endproof
In particular, for any martingale $N$ between $\mu$ and $\tilde{\nu}$, it holds $\nabla \varphi(N_t) \in C_{\bar{f}}(\nabla \varphi(N_0))$ a.s. This suggests that the sets $C_{\bar{f}}(z)$ could be described in terms of the sets $C_{\bar{g}}(x)$, as the following result confirms.
\begin{proposition}
For all $z\in \mathrm{Dom}(\bar{f})$, 
\[
C_{\bar{f}}(z) = \bigcap_{x\in \nabla \varphi^{-1}(\{z\})} \nabla \varphi(C_{\bar{g}}(x)).
\]
\end{proposition}
\proof
According to Lemma \ref{prop-lienfg}, for all $a \in \partial \bar{f}(z)$, 
\[
C^a_{\bar{f}}(z) = \{y \in \R^d : \nabla \bar{g}(a+z) =\nabla \bar{g}(a+y)\} = \{y \in \R^d : \nabla \varphi (a+z) =\nabla \varphi(a+y) + z-y\} 
\]
and $\partial \bar{f}(z)=\{ \nabla \bar{g}(x) : \nabla \varphi(x)=z\}= \{x-z : \nabla \varphi(x)=z\}= \nabla \varphi^{-1}(\{z\})-\{z\}$. 
Thus, letting $a+z=x \in \nabla \varphi^{-1}(\{z\})$, we obtain
\begin{align*}
C^a_{\bar{f}}(z) &= \{y \in \R^d : \nabla \varphi(x+y-z) = \nabla  \varphi(x)   +y-z\}\\
& = \{y \in \R^d : x+y-z \in C_{\bar{g}}(x)\} \\
&= C_{\bar{g}}(x)+ \{z-x\}\\
& = \nabla \varphi(C_{\bar{g}}(x)),
\end{align*}
since, according to Proposition \ref{prop:Cbarg},
\begin{align*}
C_{\bar{g}}(x)& =\{u\in \R^d : \nabla \varphi(u) = \nabla \varphi(x) + u-x\}\\
& = \{u\in \R^d : \nabla \varphi(u) =  z + u-x\}\\
& = \nabla \varphi(C_{\bar{g}}(x)).
\end{align*}
Finally, we get
\[
C_{\bar{f}}(z) =\bigcap_{x\in \nabla \varphi^{-1}(\{z\})} \nabla \varphi(C_{\bar{g}}(x)).
\]
\endproof

If the function $\bar{f}$ is differentiable, then similar results can be stated for the martingales involved in the transport from $\bar{\mu}$ to $\nu$.

\begin{proposition}\label{prop:Cbarf-easy}
Suppose that  $\bar{f}:\R^d \to \R$ is a differentiable dual optimizer for $\bar{\mathcal{T}}_2(\nu|\mu)$.
\begin{enumerate}
\item For all $z\in \R^d$, 
\[
C_{\bar{f}}(z) = \{u \in \R^d : \nabla \varphi^*(u) = \nabla \varphi^*(z)+u-z\},
\]
where $\varphi^* = \bar{f} + \frac{\|\,\cdot\,\|^2}{2}$ with $\nabla \varphi^*$ being the optimal transport map between $\nu$ and $\tilde{\nu}$.
\item If $(M_t)_{t\in [0,1]}$ is a martingale such that $M_0\sim \bar{\mu}$ and $M_1\sim \nu$ then $N_t = \nabla \varphi^*(M_t)$, $t\in [0,1]$, is a martingale such that $N_0 \sim \mu$ and $N_1 \sim \tilde{\nu}$.
\end{enumerate}
\end{proposition}
\proof
The proof is identical to the proofs of Propositions \ref{prop:marti-transfer} and \ref{prop:Cbarg}
 and is thus omitted.
\endproof

\section{Quadratic barycentric transport between two Gaussian measures}\label{sec:Gaussian}
If $\mu,\nu$ are gaussian probability measures on $\R^d$, the following result shows that the backward projection of $\mu$ on $\nu$ in the convex order is also a Gaussian measure, and provides a formula to calculate $\overline{{\mathcal{T}}}_2(\nu|\mu)$ in terms of their means and covariance matrices.
\begin{proposition}\label{Gauss}
If $\mu=\mathcal N(m_\mu, \Sigma_\mu)$ and $\nu=\mathcal N(m_\nu, \Sigma_\nu)$ are two Gaussian distributions on $\R^d$, then 
\begin{equation}\label{eq:Tbar-Gaussian}
    \overline{{\mathcal{T}}}_2(\nu|\mu)=\inf_{\Sigma\leq \Sigma_\nu} W_2^2(\mu,\mathcal N(m_\nu, \Sigma))
    =|m_\nu-m_\mu|^2 + \inf_{\Sigma\leq \Sigma_\nu} \mathrm{Tr}\Big(\Sigma+\Sigma_\mu-2\big(\Sigma_\mu^{1/2} \Sigma\Sigma_\mu^{1/2}\big)^{1/2}\Big),
\end{equation}
    where the infimum runs over all covariance matrices $\Sigma$ (symmetric with non negative eigenvalues), such that $\Sigma\leq \Sigma_\nu$ (which means that $\Sigma_\nu-\Sigma$ is also a covariance matrix). 
    The backward projection of $\mu$ on $\nu$ in the convex order is $\bar{\mu} = \mathcal{N}(m_\nu,\overline{\Sigma})$, with $\overline{\Sigma} \leq \Sigma_\nu$ the unique minimizer in \eqref{eq:Tbar-Gaussian}.
\end{proposition}

The proof of Proposition \ref{Gauss}, will rely on the following two classical lemmas.
\begin{lemma}\label{lem1} Let $\nu=\mathcal N(m_\nu, \Sigma_\nu)$ and $\eta=\mathcal N(m, \Sigma)$ then $\eta\leq_c \nu$ if and only if $m=m_\nu$ and $\Sigma\leq \Sigma_\nu$.  
\end{lemma}

\begin{lemma}[Theorem 2.1, \cite{cuesta-albertosLowerBounds2Wasserstein1996}]\label{lem2} 
For any $\eta,\nu \in \mathcal{P}_2(\R^d)$, 
\[
W_2(\eta,\mu)\geq W_2\big(\mathcal N(m_\eta, \Sigma_\eta),\mathcal N(m_\mu, \Sigma_\mu)\big).
\]
\end{lemma}

\begin{proof}[Proof of Proposition \ref{Gauss}]
    The  second equality equality in \eqref{eq:Tbar-Gaussian} directly follows from the following well known expression for the Wasserstein distance between two Gaussian distributions:  (see e.g. \cite[Theorem 2.2]{Tak11}),
    \[
    W_2^2(\mathcal N(m, \Sigma),\mathcal N(m', \Sigma'))= |m-m'|^2 +\mathrm{Tr}\Big(\Sigma+   \Sigma'-2\big(\Sigma'^{1/2} \Sigma\Sigma'^{1/2}\big)^{1/2}\Big).
    \]
    Let us prove the first equality in \eqref{eq:Tbar-Gaussian}. According to \eqref{eq:GJ20} and  applying    Lemma \ref{lem2} and then Lemma \ref{lem1} yields
\begin{align*}
\overline{\mathcal{T}}_2(\nu|\mu)=\inf_{\eta\leq_c \nu } W_2^2(\mu,\eta)
\geq \inf_{\eta\leq_c \nu } W_2^2(\mu,\mathcal N
(m_\eta, \Sigma_\eta))
= \inf_{\Sigma\leq \Sigma_\nu }W_2^2(\mu,\mathcal N
(m_\nu, \Sigma))
&= \inf_{\eta' \text{ Gaussian},\ 
 \eta'\leq_c \nu }W_2^2(\mu,\eta'
)\\
&\geq \overline{\mathcal{T}}_2(\nu|\mu).
\end{align*}
Therefore, equality holds at every step.
\end{proof}

For the sake of completeness, we include the proof of Lemma \ref{lem1}.
\begin{proof}[Proof of Lemma \ref{lem1}] Let $\nu=\mathcal N(m_\nu, \Sigma_\nu)$ and $\eta=\mathcal N(m, \Sigma)$. Assume $m=m_\nu$ and $\Sigma\leq \Sigma_\nu$. If $X$ and $Y$ are two independent random vectors with respective laws $\mathcal \mc N(m,\Sigma)$ and $\mc N(0,\Sigma_\nu-\Sigma)$, then setting $Z=X+Y$, one easily checks that $(X,Z)$ is a martingale and that $Z\sim N(m_\nu,\Sigma_\nu)$. As a consequence, according to (the easy case) of Strassen's characterization of convex order \cite{Strassen}, one has $\eta\leq_c \nu$. Conversely if $\eta\leq_c \nu$ then  for any real convex function $f$,
\[\int f \,\ud\eta\leq \int f \,\ud\nu.\]
Applying this inequality with $f(x)=x_i$ or $f(x)=-x_i$, $i\in \{1,\ldots,d\}$, gives $m=m_\nu$. And choosing $f(x)=\langle x, v\rangle^2$ with  $v\in \R^d$ provides $\Sigma\leq\Sigma_\nu$.
\end{proof}

In the case where $\Sigma_\mu$ and $\Sigma_\nu$ commute, the backward projection $\bar{\mu}$ is easy to identify, as shows the following result. If $D,D'$ are two diagonal matrices, we will denote $\min(D,D') = [\min(D_{i,j},D'_{i,j})]_{1\leq i,j \leq d}$.  
\begin{proposition}\label{prop-Gaussian-Commute}
    Suppose that $\Sigma_\mu\Sigma_\nu=\Sigma_\nu\Sigma_\mu$ and let $P$ be an orthonormal matrix $P$ such that $D_\mu:=P\Sigma_\mu P^T$ and  $D_\nu:=P\Sigma_\nu P^T$ are both diagonal. Setting $\bar{D}:=\min(D_\mu, D_\nu)$, the backward projection of $\mu=\mc N(m_\mu,\Sigma_\mu)$ on $\nu=\mc N(m_\nu,\Sigma_\nu)$ is $\bar \mu = \mc N(m_\nu,\bar \Sigma)$, with $\bar \Sigma=P^T \bar{D} P$.
\end{proposition}

The proof of this result will make use of the following Lemma. 
\begin{lemma}\label{marseille17-06}
    For any symmetric positive semidefinite  matrix $\Gamma$,
   \[\mathrm{Tr}\big(\sqrt \Gamma\big)\leq \mathrm{Tr}\big(\sqrt {D_\Gamma} \big),\] 
   where $D_\Gamma$ denotes the diagonal matrix whose diagonal is the one of $\Gamma$.    
\end{lemma}
\begin{proof}[Proof of Lemma \ref{marseille17-06}]
Let $\Gamma$ be a symmetric positive semidefinite  matrix and let $S=\Gamma^{1/2}$ be its square root. Then one has 
\[
\Gamma_{ii}=\sum_{j=1}^d S_{ij}^2=S_{ii}^2 +\sum_{j\neq i} S_{ij}^2 \geq S_{ii}^2.
\]
It follows that
\[
\mathrm{Tr}\big(\sqrt \Gamma\big)=\sum_{i=1}^d S_{ii}
\leq \sum_{i=1}^d\sqrt{\Gamma_{ii}}=\mathrm{Tr}\big(\sqrt {D_\Gamma}\big).
\]
\end{proof}

\proof[Proof of Proposition \ref{prop-Gaussian-Commute}.]
Without loss of generality, we can assume that $\Sigma_\mu=D_\mu$ and $\Sigma_\nu=D_\nu$ are diagonal, with diagonal elements denoted $\sigma_i^2(\mu)$ and $\sigma_i^2(\nu)$, $1\leq i\leq d$. According to Proposition \ref{Gauss}, $\bar \mu = \mc N(m_\nu,\bar \Sigma)$, with $\bar \Sigma$ a minimizer of 
\[
\mathcal{F}(\Sigma) = \mathrm{Tr}\left(\Sigma+D_\mu-2\left(D_\mu^{1/2} \Sigma D_\mu^{1/2}\right)^{1/2}\right),
\]
over the $\mc S_\nu:=\{\Sigma : 0\leq \Sigma \leq D_\nu\}$.
To prove that the minimizer is attained at a diagonal matrix, it is enough to show that for every $\Sigma \in \mc S_\nu$, the diagonal matrix $D_\Sigma$ obtained by canceling all the non-diagonal terms of $\Sigma$ is such that $\mc F(\Sigma) \geq \mc F(D_\Sigma)$. Note that $D_\Sigma \in \mc S_\nu$, since $0\leq\Sigma\leq D_\nu$ implies that for all $1\leq i\leq d$, $0\leq\Sigma_{ii}\leq \sigma_i^2(\nu)$ and so $0\leq D_{\Sigma}\leq D_\nu$.
Applying Lemma \ref{marseille17-06} to $\Gamma:=D_\mu^{1/2} \Sigma D_\mu^{1/2}$ for which $D_\Gamma = D_\mu^{1/2} D_\Sigma D_\mu^{1/2}$, yields
\[
\mathrm{Tr}\left(\left(D_\mu^{1/2} \Sigma D_\mu^{1/2}\right)^{1/2}\right) \leq \mathrm{Tr}\left(\left(D_\mu^{1/2} D_\Sigma D_\mu^{1/2}\right)^{1/2}\right)
\]
and so
\[
\mc F(\Sigma) =\mathrm{Tr}(\Sigma)+\mathrm{Tr}(D_\mu)-2\mathrm{Tr}\left(\left(D_\mu^{1/2} \Sigma D_\mu^{1/2}\right)^{1/2}\right) \geq \mathrm{Tr}(D_\Sigma)+\mathrm{Tr}(D_\mu)-2\mathrm{Tr}\left(\left(D_\mu^{1/2} D_\Sigma D_\mu^{1/2}\right)^{1/2}\right)=\mc F(D_\Sigma),
\]
which proves the claim. We conclude that $\bar{\Sigma}$ is a diagonal matrix. Now, if $D$ is a diagonal matrix with diagonal elements denoted $\lambda_1,\ldots,\lambda_d$, we have $D\in S_\nu$ if and only if $0\leq \lambda_i \leq \sigma_i^2(\nu)$ for all $i$, and in this case,
\[
\mc F(D) = \sum_{i=1}^d\left(\lambda_i+\sigma_i^2(\mu)-2\lambda_i^{1/2}\sigma_i(\mu)\right) = \sum_{i=1}^d\left(\sigma_i(\mu)-\lambda_i^{1/2}\right)^2 \geq \sum_{i=1}^d [\sigma_i(\mu)-\sigma_i(\nu)]_+^2,
\]
with equality if, and only if, $\lambda_i=\min(\sigma^2_i(\mu), \sigma_i^2(\nu))$, that is if $D=\min(D_\mu,D_\nu)=\bar{\Sigma}$.
\endproof
An immediate consequence of the proof above is the following closed formula for the quadratic barycentric optimal transport cost between two Gaussian laws with diagonal covariance matrices.
\begin{corollary}If $\mu=\mc N(m_\mu,\Sigma_\mu)$ and $\nu=\mc N(m_\nu,\Sigma_\nu)$ with $\Sigma_\mu$ and $\Sigma_\nu$ two diagonal matrices with diagonal elements denoted $\sigma_i^2(\mu)$ and $\sigma_i^2(\nu)$, $1\leq i\leq d$, then  
\begin{equation}\label{gaussdiag}
\overline{{\mathcal{T}}}_2(\nu|\mu)=|m_\nu-m_\mu|^2+\sum_{i=1}^d[\sigma_i(\mu)-\sigma_i(\nu)]_+^2.
\end{equation}
\end{corollary}
    
Let us do a few observations to complete the picture and connect with the results of the first parts. Suppose that $\mu = \mc N(m_\mu, \Sigma_\mu)$ and  $\nu = \mc N(m_\nu, \Sigma_\nu)$ and let $\bar \Sigma$ be the unique solution of \eqref{eq:Tbar-Gaussian}. We assume, for simplicity, that $\Sigma_\mu$ and $\bar{\Sigma}$ are invertible. 
\begin{itemize}
    \item The optimal transport map $\nabla \varphi$ sending $\mu$ onto $\bar \mu = \mc N(m_\nu,\bar\Sigma)$ is the linear map
    \[
    \nabla \varphi(x) = Ax,\qquad \text{with}\qquad A=\bar{\Sigma}^{1/2}(\bar{\Sigma}^{1/2}\Sigma_\mu \bar{\Sigma}^{1/2})^{-1/2} \bar{\Sigma}^{1/2}.
    \]
    It corresponds to the quadratic form $\varphi(x) = \frac12 \langle x, Ax\rangle$, $x\in \R^d.$
    \item Since $A$ is invertible, the conjugate of $\varphi$ is given by $\varphi^*(y) = \frac12 \langle y, A^{-1}y\rangle$, $y\in \R^d$. The associated dual optimizer is $\bar{f}(y)= \frac12 \langle y, A^{-1}y\rangle + \frac12 \|y\|^2$, $y\in \R^d$.
    \item Since $\varphi^*$ is smooth, one can apply Proposition \ref{prop:Cbarf-easy}, and conclude that, for all $z\in \R^d$,
    \[
    C_{\bar{f}}(z) = \{z\}+ \mathrm{Ker}(A^{-1}-I_d) = \{z\}+ \mathrm{Ker}(A-I_d).
    \]
    Since this is an affine subspace of $\R^d$, one gets $D_{\bar{f}}(z)=C_{\bar{f}}(z)$, for all $z\in \R^d$.
    In other words, the martingales transporting $\bar \mu$ on $\nu$ move only along the directions that are left fixed by the transport map $A$ from $\mu$ to $\bar{\mu}$.
    \item If $B=(B_t)_{t\in [0,1]}$ is a standard Brownian motion, and $X_0$ is a random vector with law $\mu$ independent of $B$, then
    \[
    M_t = AX_0 + (\Sigma_\nu-\bar \Sigma)^{1/2} B_t
    \]
    is a martingale such that $M_0 \sim \bar{\mu}$ and $M_1 \sim \nu$.
    The process $X$ defined by
    \[
    X_t=(1-t)X_0 +tAX_0+(\Sigma_\nu-\bar \Sigma)^{1/2} B_t,\qquad t\in [0,1]
    \]
    is a geodesic from $\mu$ to $\nu$ in the sense of Theorem \ref{thm:BBbar}.
    \item Combining the last two points, we conclude that the matrix $\bar{\Sigma}$ is such that 
    \[
    \mathrm{Im}(\Sigma_\nu-\bar{\Sigma}) \subset \mathrm{Ker}(A-I_d).
    \]
    or equivalently $(A-I_d)\times (\Sigma_\nu-\bar{\Sigma}) = (\Sigma_\nu-\bar{\Sigma})(A-I_d)=0$. This condition can be alternatively recovered as a first order optimality condition for the optimization problem \eqref{eq:Tbar-Gaussian}. We omit details.
    Let us check this condition in the particular case where $\Sigma_\mu$ and $\Sigma_\nu$ are diagonal and invertible.  Then, as shown above, $\bar{\Sigma} = \min(\Sigma_\mu,\Sigma_\nu)$. So $A=\bar{\Sigma}^{1/2} \Sigma_\mu^{-1/2} = \min(I_d, \Sigma_\nu^{1/2}\Sigma_\mu^{-1/2})$. On the other hand, $\Sigma_\nu-\bar{\Sigma}= \max(\Sigma_\nu-\Sigma_\mu,0)$. So we see that $\mathrm{Im}(\Sigma_\nu-\bar{\Sigma})= \mathrm{Ker}(A-I_d)$.
\end{itemize}

To identify the backward projection, one can use the following result.
\begin{theorem}{(\cite[Theorem 2.1]{gozlan2020mixture})}\label{thm:GJ20}
    The following are equivalent:
    \begin{itemize}
        \item $\nu = \bar{\mu}$
        \item The Brenier transport map $T$ sending $\mu$ onto $\nu$ is $1$-Lipschitz.
    \end{itemize}
\end{theorem}

\begin{proposition}\label{prop:barmu=nu}
    Suppose that $\mu=\mathcal{N}(0,\Sigma_\mu)$ and $\nu=\mathcal{N}(0,\Sigma_\nu)$. If $\Sigma_\mu$ is invertible and $\Sigma_\nu \leq \Sigma_\mu$, then $\bar{\mu}=\nu$.
\end{proposition}

The proof follows immediately from Theorem \ref{thm:GJ20} and  the following lemma.
\begin{lemma}\label{lem:convexorder->contraction}
Suppose that $\mu=\mathcal{N}(0,\Sigma_\mu)$ and $\nu=\mathcal{N}(0,\Sigma_\nu)$. If $\Sigma_\mu$ is invertible and $\Sigma_\nu \leq \Sigma_\mu$, then the Brenier transport map $T$ sending $\mu$ onto $\nu$ is $1$-Lipschitz.
\end{lemma}
\proof
The map $T$ is linear: $T(x)=Ax$, $x\in \R^d$, with $A$ symmetric positive semidefinite matrix such that $A\Sigma_\mu A=\Sigma_\nu$.
By assumption, one gets
\[
A\Sigma_\mu A\leq \Sigma_\mu.
\]
If $x$ is an eigenvector of $A$ with associated eigenvalue $\lambda$, then
\[
x^T\Sigma_\mu x\geq x^T A\Sigma_\mu Ax=\lambda^2x^T\Sigma_\mu x,
\]
which shows that $\lambda^2\le 1$ since $\Sigma_\mu$ is positive definite, and thus $0\le \lambda\le 1$ since $0\leq A$. 
\endproof

\begin{remark}
According to Lemma \ref{lem1}, the condition $\Sigma_\nu \leq \Sigma_\mu$ is equivalent to $\nu \leq_c \mu$. One could thus ask, if the conclusion of Proposition \ref{prop:barmu=nu}, could be extended to general probability measures, under the condition $\nu \leq_c \mu$. This is not the case. Namely, take $\mu=\frac{1}{4}\delta_{-2}+\frac{1}{4}\delta_{2}+\frac{1}{2}\delta_{0}$ and $\nu=\frac{1}{2}\delta_{-1}+\frac{1}{2}\delta_{1}$. One has $\nu\le_c\mu$ but there is no transport map sending $\mu$ onto $\nu$.
\end{remark}

Let us finally discuss the conclusion of Lemma \ref{lem:convexorder->contraction}. 

First, observe that the converse implication is false in general: there exist $\Sigma_\mu$ and $\Sigma_\nu$ such that $T$ is $1$-Lipschitz but $\Sigma_\nu\nleq\Sigma_\mu$. Namely, take $T(x)=Ax$, $x\in \R^d$, with
\[
A=\begin{bmatrix}
1 - \delta & 0 \\
0 & 1
\end{bmatrix}\qquad \text{and}\qquad \Sigma_\mu=\begin{bmatrix}
1 + \epsilon & \epsilon - 1 \\
\epsilon - 1 & 1 + \epsilon
\end{bmatrix},
\]
for some $\delta,\epsilon \in (0,1)$.
A simple calculation shows that
\[
\Sigma_\nu=\begin{bmatrix}
(1 - \delta)^2 (\epsilon + 1) & (1 - \delta)(\epsilon - 1) \\
(1 - \delta)(\epsilon - 1) & \epsilon + 1
\end{bmatrix}
\]
and 
\[
\Sigma_\nu-\Sigma_\mu=\begin{bmatrix}
(1 - \delta)^2 (\epsilon + 1) - \epsilon - 1 & - \delta(\epsilon - 1) - \epsilon + 1 \\
-\delta(\epsilon - 1) - \epsilon + 1 & 0
\end{bmatrix}
\]
whose determinant is negative.
\begin{remark}
This example also shows that the projection of $\mu$ onto the measures $\le_c \nu$ is not the same as the projection of $\nu$ onto the measures $\le_c \mu$, i.e. the problem is not symmetric in $(\mu,\nu)$. Indeed, for this example, since $T$ is $1$-Lipschitz, according to Theorem \ref{thm:GJ20}, the projection of $\mu$ onto $\{ \eta :  \eta \le_c \nu\}$ is $\nu$ itself.  However, $\nu$ cannot be the projection of $\nu$ onto $\{ \eta : \eta \le_c \mu\}$, since $\Sigma_\nu\npreceq\Sigma_\mu$ and so $\nu \nleq_c \mu$. 
\end{remark}

Nevertheless, one has the following partial converse to Lemma \ref{lem:convexorder->contraction}, under very specific hypotheses.
\begin{proposition}\ 
\begin{enumerate}
\item Let $\mu=\mathcal{N}(0,\Sigma_\mu)$ and $\nu=\mathcal{N}(0,\Sigma_\nu)$ with $\Sigma_\mu$  invertible and $\Sigma_\nu\Sigma_\mu=\Sigma_\mu\Sigma_\nu$. Then the Brenier transport map $T$ sending $\mu$ onto $\nu$ is $1$-Lipschitz if and only if $\Sigma_\nu \leq \Sigma_\mu$.
\item Let $\mu=\mathcal{N}(0,I_d)$ be the standard Gaussian measure and $\nu \in \mc P_2(\R^d)$ a centered probability measure. If the Brenier map $T$ sending $\mu$ onto $\nu$ is $1$-Lipschitz, then $\nu \leq_c \mu$.
\end{enumerate}
\end{proposition}
\begin{proof}
\begin{enumerate}
\item In this commuting case, $T(x) = \Sigma_\nu^{1/2}\Sigma_\mu^{-1/2} x$, $x\in \R^d$, so that the conclusion is straightforward.
\item The second point is a straightforward adaptation of  \cite{harge2001inequalities}. Details are left to the reader.
\end{enumerate}
\end{proof}

\end{document}